\documentclass[11pt]{amsart}

\usepackage{amsmath}
\usepackage{amstext}
\usepackage{amssymb}
\usepackage{amsthm}
\usepackage{enumerate,dsfont,graphics,color}
\usepackage{cite}

\makeatletter
\def\imod#1{\allowbreak\mkern10mu({\operator@font mod}\,\,#1)}
\makeatother

\newtheorem{theorem}{Theorem}[section]

\newtheorem*{mainthm}{Main Theorem}
\newtheorem*{mainlemma}{Main Lemma}

\newtheorem{lemma}[theorem]{Lemma}
\newtheorem{claim}[theorem]{Claim}

\newtheorem*{clm}{Claim}
\newtheorem{corollary}[theorem]{Corollary}

\theoremstyle{definition}
\newtheorem{definition}[theorem]{Definition}
\newtheorem{example}[theorem]{Example}

\newtheorem{question}[theorem]{Question}
\newtheorem{context}[theorem]{Context}
\theoremstyle{remark}
\newtheorem{remark}[theorem]{Remark}

\theoremstyle{remark}

\numberwithin{equation}{section}

    \DeclareMathOperator{\dom}{dom}

    \DeclareMathOperator{\rank}{rank}

    \DeclareMathOperator{\Dp}{Dp}
    
    \DeclareMathOperator{\non}{non}
    \DeclareMathOperator{\add}{add}
    \DeclareMathOperator{\cof}{cof}
    \DeclareMathOperator{\limdir}{limdir}
    \DeclareMathOperator{\cov}{cov}

\newcommand{\restrict}{{\upharpoonright}}
\newcommand{\rest}{{\upharpoonright}}
\newcommand{\frestr}{\!\!\upharpoonright\!\!}

    \newcommand{\thzfc}{\mathrm{ZFC}}

    \newcommand{\la}{\langle}
    \newcommand{\ra}{\rangle}
    \newcommand{\PO}{\mathcal{P}}
    \newcommand{\menos}{\smallsetminus}

    \newcommand{\afrak}{\mathfrak{a}}
    \newcommand{\bfrak}{\mathfrak{b}}
    \newcommand{\cfrak}{\mathfrak{c}}
    \newcommand{\dfrak}{\mathfrak{d}}
    
    \newcommand{\gfrak}{\mathfrak{g}}
    
    \newcommand{\pfrak}{\mathfrak{p}}
    \newcommand{\rfrak}{\mathfrak{r}}
    \newcommand{\sfrak}{\mathfrak{s}}
    
    \newcommand{\ufrak}{\mathfrak{u}}

    \newcommand{\Awf}{\mathcal{A}}
    \newcommand{\Bwf}{\mathcal{B}}
    \newcommand{\Cwf}{\mathcal{C}}
    
    \newcommand{\Ewf}{\mathcal{E}}
    
    \newcommand{\Gwf}{\mathcal{G}}
    \newcommand{\Hwf}{\mathcal{H}}
    \newcommand{\Iwf}{\mathcal{I}}
    \newcommand{\Jwf}{\mathcal{J}}
    \newcommand{\Mwf}{\mathcal{M}}
    \newcommand{\Nwf}{\mathcal{N}}
    
    \newcommand{\Pwf}{\mathcal{P}}
    \newcommand{\Swf}{\mathcal{S}}

    \newcommand{\Bor}{\mathbb{B}}
    \newcommand{\Cor}{\mathbb{C}}
    \newcommand{\Dor}{\mathbb{D}}
    
    \newcommand{\Loc}{\mathbb{LOC}}
    \newcommand{\Mor}{\mathbb{M}}
    \newcommand{\Por}{\mathbb{P}}
    \newcommand{\Qor}{\mathbb{Q}}
    
    \newcommand{\Sor}{\mathbb{S}}

    \newcommand{\Qnm}{\dot{\mathbb{Q}}}
    
    \newcommand{\Snm}{\dot{\mathbb{S}}}

    \newcommand{\Bnm}{\dot{\mathbb{B}}}

    \newcommand{\Locnm}{\dot{\mathbb{LOC}}}

\newcommand{\sii}{{\ \mbox{$\Leftrightarrow$} \ }}

\date{\today}

\begin{document}

\title[Splitting, Bounding and Almost Disjointness can be quite Different]{Splitting, Bounding, and Almost Disjointness can be quite Different}

\author{Vera Fischer}

\address{Institut f\"ur Diskrete Mathematik und Geometrie, Technishe Universit\"at Wien, Wiedner Hauptstrasse 8-10/104, 1040 Wien, Austria}
\curraddr{Kurt G\"odel Research Center, University of Vienna, W\"ahringer Strasse 25, 1090 Vienna, Austria}
\email{vera.fischer@univie.ac.at}

\author[Diego A. Mej\'ia]{Diego Alejandro Mej\'ia}

\address{Institut f\"ur Diskrete Mathematik und Geometrie, Technishe Universit\"at Wien, Wiedner Hauptstrasse 8-10/104, 1040 Wien, Austria}
\curraddr{Department of Mathematics, Shizuoka University, Ohya 836, Shizuoka, 422-8529 Japan}
\email{diego.mejia@shizuoka.ac.jp}

\thanks{The first author thanks the Austrian Science Fund (FWF) for the generous support
through grant no. M1365-N13. The second author thanks the support from FWF through grant no. P23875-N13 and I1272-N25 and the support of the Monbukagakusho (Ministry of Education, Culture, Sports, Science and Technology of Japan) Scholarship.}



\begin{abstract} We prove the consistency of $\add(\Nwf)<\cov(\Nwf)<\pfrak=\sfrak=\gfrak<\add(\Mwf)=\cof(\Mwf)<\afrak=\rfrak=\non(\Nwf)=\cfrak$ with $\thzfc$ where each of these cardinal invariants assume arbitrary uncountable regular values.
\end{abstract}

\maketitle


\section{Introduction}\label{SecIntro}

The splitting, the bounding and the almost disjointness numbers, denoted $\mathfrak{s}$, $\mathfrak{b}$ and $\mathfrak{a}$ respectively,
have been of interest for already a long time. The splitting and the bounding numbers, as well as the splitting and the almost disjointness numbers, are independent, while a not difficult ZFC argument shows that $\mathfrak{b}\leq\mathfrak{a}$ (see~\cite{blass}). The consistency of $\mathfrak{s}<\mathfrak{b}=\mathfrak{a}$ holds in the
Hechler model (see~\cite{baudor}). In 1984, introducing the powerful technique of {\emph{creature forcing}}, S. Shelah~\cite{SSCR} obtained a generic extension in which cardinals are not collapsed and $\mathfrak{b}=\aleph_1<\mathfrak{a}=\mathfrak{s}=\aleph_2$. As this is a countable support iteration of proper forcing argument (thus, restricted to force $\cfrak$ at most $\aleph_2$), it remained interest to generalize these results on models of larger continuum, i.e. models of $\mathfrak{c}>\aleph_2$. Almost 15 years later, J. Brendle~\cite{brendle98} showed that
consistently $\mathfrak{b}=\kappa<\mathfrak{a}=\kappa^+$, while in 2008 the first author jointly with J. Stepr{\=a}ns~\cite{VFJS} obtained the consistency of $\mathfrak{b}=\kappa<\mathfrak{s}=\kappa^+$, where $\kappa$ is an arbitrary regular uncountable cardinal. Even though the constructions
can be combined to produce the consistency of $\mathfrak{b}=\kappa<\mathfrak{a}=\mathfrak{s}=\kappa^+$, they can not be further generalized to produce a model in which there is an arbitrarily large spread between the relevant cardinal characteristics.

To show the consistency of $\aleph_1<\mathfrak{d}<\mathfrak{a}$ (without the assumption of a measurable), where $\mathfrak{d}$ is the dominating number,
S. Shelah~\cite{shelah04} introduced a ground-breaking, new technique, known as {\emph{template iterations}}. Since this technique is central to the current paper, we will add a few more lines regarding this construction. In his work, Shelah generalizes the classical fsi (finite support iteration) of Suslin ccc posets to the context of a finite-supported iteration along an arbitrary linear order, where the iteration is constructed from a well-founded structure of subsets of the linear order, known as a \emph{template}. As an application, assuming $\mathrm{CH}$ and $\aleph_1<\mu<\lambda$ regular cardinals with $\lambda^{\aleph_0}=\lambda$, he constructs a template so that the iteration using Hechler forcing (the standard ccc poset adding a dominating real) along this template produces a $\mu$-scale in the extension to get $\bfrak=\dfrak=\mu$ and, on the other hand, by an isomorphism-of-names argument, there are no mad (maximal almost disjoint) families of size between $\mu$ (including it) and $\lambda$ (excluding it), so $\afrak=\cfrak=\lambda$ in the extension (because $\bfrak\leq\afrak$). In this model $\mathfrak{s}=\aleph_1$ and so all of $\mathfrak{s}$, $\mathfrak{b}$ and $\mathfrak{a}$ are distinct in Shelah's template extension. The same consistency result was obtained for $\lambda$ singular with uncountable cofinality and, later, for instances of $\lambda$ of countable cofinality by Brendle \cite{brendle03}.

In \cite{VFJB11}, using a method known as {\emph{matrix iteration}}, the first author jointly with J. Brendle, established the consistency of $\mathfrak{a}=\mathfrak{b}=\kappa<\mathfrak{s}=\lambda$, where $\kappa<\lambda$ are arbitrary regular uncountable cardinals. This result heavily depends on a new method of preserving the maximality of a certain maximal almost disjoint family along such an iteration. In the same paper, it is shown that $\mathfrak{b}=\kappa<\mathfrak{s}=\mathfrak{a}=\lambda$, where $\kappa$ is above a measurable in the ground model, thus generalizing Shelah's creature posets result mentioned earlier. The authors ask if any of the following two constellations $\mathfrak{b}<\mathfrak{a}<\mathfrak{s}$, as well as $\mathfrak{b}<\mathfrak{s}<\mathfrak{a}$ are consistent. Both of those remain very interesting open questions.

As an attempt to get a model of $\aleph_1<\sfrak<\bfrak<\afrak$, the second author~\cite{mejia-temp} introduced the {\emph{iteration of non-definable ccc posets along a template}}. He proved that if $\theta<\kappa<\mu<\lambda$ are uncountable regular cardinals, $\kappa$ is measurable, $\theta^{<\theta}=\theta$ and $\lambda^\kappa=\lambda$, then there is a ccc poset forcing $\sfrak=\pfrak=\gfrak=\theta$, $\bfrak=\dfrak=\mu$ and $\afrak=\cfrak=\lambda$. Also, $\non(\Nwf)=\rfrak=\lambda$ and (by a slight modification of the forcing) $\add(\Nwf)=\cov(\Nwf)=\theta$ hold in the extension.
The forcing construction is a matrix iteration involving parallel template iterations, as in Shelah's original template model, modulo a measurable cardinal.

In this paper we show that consistently $\aleph_1<\mathfrak{s}<\mathfrak{b}<\mathfrak{a}$ without the assumption of a measurable, which solves \cite[Question 8.1]{mejia-temp}. In addition, answering~\cite[Question 8.2]{mejia-temp}, we show that given arbitrary regular uncountable cardinals  $\theta_0<\theta_1<\theta<\mu<\lambda$, there is a ccc generic extension in which  $\add(\Nwf)=\theta_0<\add(\Nwf)=\theta_1<\pfrak=\sfrak=\gfrak=\theta<\add(\Mwf)=\cof(\Mwf)=\mu<\afrak=\rfrak=\non(\Nwf)=\cfrak$.

First, we want to address the consistency of  $\aleph_1<\mathfrak{s}=\theta<\mathfrak{b}=\mu<\mathfrak{a}=\lambda$ (all regular cardinals) without the assumption of a measurable.
Let  $\la L^\lambda,\bar{\Iwf}^\lambda\ra$ denote the template used in Shelah's original consistency proof of $\mathfrak{d}<\mathfrak{a}$. To obtain the
desired constellation, it seems natural to iterate along $\la L^\lambda,\bar{\Iwf}^\lambda\ra$ Hechler forcing for adding a dominating real and Mathias-Prikry posets used to guarantee that $\mathfrak{s}=\theta$. To force $\theta\leq\pfrak \;(\leq\sfrak)$, we use Mathias-Prikry posets (of size $<\theta$) to add a pseudo-intersection to every filter base of size $<\theta$ (by a quite standard counting argument adapted to the context of template iterations). To force $\sfrak\leq\theta$ we aim to preserve a splitting family of size $\theta$ that is generated in some middle step of the iteration (actually, this splitting family is formed by $\theta$-many Cohen reals). The preservation results from \cite[Sect. 5]{mejia-temp} and the fact that Hechler forcing preserves some sort of splitting families (see \cite{baudor}) provide $\aleph_1<\sfrak<\bfrak<\cfrak$. However, with the use of Mathias-Prikry posets, the construction is not uniform enough for an isomorphism of names argument to go through and it is not clear how to provide $\mathfrak{b}<\mathfrak{a}$. Noticing that Shelah's template $\la L^\lambda,\bar{\Iwf}^\lambda\ra$  is not only equipped with a length but with a \emph{width}, we construct a poset by recursion on the width in such a way that small mad families are eliminated at successor steps. To be more precise, for $\delta\leq\lambda$, let $\la L^\delta,\bar{\Iwf}^\delta\ra$ be Shelah's template with width $\delta$ (see Section \ref{SecMain}). We construct an increasing sequence of template iterations (using Hechler forcing and Mathias-Prikry posets) along these templates by recursion on $\delta$. In the successor steps, we expand the iteration along $\la L^\delta,\bar{\Iwf}^\delta\ra$ to an iteration along $\la L^{\delta'},\bar{\Iwf}^{\delta'}\ra$ for some $\delta'\in(\delta,\lambda)$ such that one a.d. (almost disjoint) family of size $\nu\in[\mu,\lambda)$ in the generic extension at $\delta$ is not mad in the generic extension at $\delta'$. By a book-keeping device for these a.d. families, the iteration along $\la L^\lambda,\bar{\Iwf}^\lambda\ra$, being the direct limit of the previous iterations, forces that either $\afrak=\lambda$ or $\afrak<\mu$ (but, as we aim to force $\bfrak=\mu$, the only option would be $\afrak=\lambda$).

In order, to achieve the above recursive construction, we need a better understanding of isomorphims between generalized template iterations, i.e. iterations along a template which involve non-definable iterands (see  Lemma \ref{InnEqv}). It is known that two template iterations of Hechler poset are isomorphic if the template structures are isomorphic (or just innocuously different, as described in Definition \ref{DefInnoc}), which is not the case when non-definable posets are used in the iteration.
In addition we need to work with an extended notion of isomorphism between subsets of the underlying template of generalized template iterations, see Definition~\ref{DefItIsom}.


The previous construction can be modified in a natural way to construct a model of $\aleph_1<\add(\Nwf)<\cov(\Nwf)<\sfrak<\bfrak<\afrak$, but in order to preserve witnesses for $\add(\Nwf)$, $\cov(\Nwf)$ and $\sfrak$ (simultaneously)  we need to further develop some already existing preservation results regarding template iterations. There are two such results, which are of interest for us: Theorems 5.8 and 5.10 from \cite{mejia-temp}. The first of those theorems can not be applied to preserve witnesses of different size along the same iteration, for example, to preserve a witness to $\cov(\Nwf)$ which is smaller than a witness of $\sfrak$. The second theorem can be applied to standard fsi's when they are viewed as template iterations, which is the reason why additional simpler consistency results, including the groupwise density number, $\gfrak$, were obtained in \cite{mejia-temp}. However, we do not know if this second preservation theorem can be applied to obtain the consistency of $\aleph_1<\add(\Nwf)<\cov(\Nwf)<\sfrak<\bfrak<\afrak$ modulo a measurable.  In view of this, one important achievement of this paper is that the second preservation theorem (Theorem \ref{PresTemp2}) works for iterations along Shelah's template, see Lemma 5.8 and Theorem \ref{ThmItPres}.

Relaying on this new preservation theorem, Theorem~\ref{ThmItPres}, we can show that a certain class of template iterations, to which we refer as {\emph{pre-appropriate iterations}} (see Definition~\ref{DefAppr}, parts (1)-(7)), can preserve witnesses for $\add(\Nwf)\leq\theta_0$, $\cov(\Nwf)\leq\theta_1$ and $\sfrak\leq\theta$. In addition, our pre-appropriate iterations force that $\add(\Mwf)=\cof(\Mwf)=\mu$ and $\gfrak=\theta$, the latter by an argument that already appears in \cite{mejia-temp} using Lemma \ref{lemmagfrak} (originally by Blass \cite{Bl}) and Theorem \ref{newrealint}.  Now, by a consequence (Theorem \ref{dfrakbig}) of the first preservation theorem above, we show that in generic extensions obtained via pre-appropriate iterations, $\rfrak=\non(\Nwf)=\cfrak=\lambda$. In addition, we can guarantee that our iterations provide lower bounds for $\add(\Nwf)$, $\cov(\Nwf)$ and $\pfrak \;(\leq\sfrak)$ (see the notion of {\emph{appropriate iteration}} and clauses (8)-(10) of Definition~\ref{DefAppr}). Thus in generic extensions obtained via appropriate iterations, all cardinal characteristics, except the almost disjointness number, have the desired values (see Lemma \ref{ApprLemma}). The methods which provide that $\mathfrak{a}=\lambda$ in our final extension, were already discussed earlier. Thus we can state our main result:

\begin{mainthm}
   Let $\theta_0\leq\theta_1\leq\theta<\mu<\lambda$ be uncountable regular cardinals with $\theta^{<\theta}=\theta$ and $\lambda^{<\lambda}=\lambda$. Then, there is a ccc poset that forces $\add(\Nwf)=\theta_0$, $\cov(\Nwf)=\theta_1$, $\pfrak=\sfrak=\gfrak=\theta$, $\add(\Mwf)=\cof(\Mwf)=\mu$ and $\afrak=\non(\Nwf)=\rfrak=\cfrak=\lambda$.
\end{mainthm}

This paper is structured as follows. Sections \ref{SecPre} and \ref{SecTemp} contain preliminary knowledge of the paper, the latter section presented as a summary of the template iteration theory in \cite[Sect. 3 and 4]{mejia-temp}. Additionally, we discuss in Section \ref{SecTemp} isomorphisms of template iterations. Section \ref{SecSelahTemp} defines Shelah's templates and explains those features, which are useful for our isomorphism-of-names arguments in the context of template iterations with non-definable posets. In Section \ref{SecPresProp} we develop the preservation theory for iterations along Shelah's templates. Section \ref{SecMain} is devoted to the proof of the Main Theorem and Section \ref{SecDisc} contains some open questions.



\section{Preliminaries}\label{SecPre}

\subsection{Classical cardinal invariants}\label{SubSecCard}
This section contains some definitions and basic facts regarding the cardinal characteristics of the continuum which we are to consider. Further information about them can be found, for example, in \cite{judabarto} and \cite{blass}.

For $f,g\in\omega^\omega$, we said that \emph{$f$ is eventually dominated by $g$}, denoted $f\leq^*g$, if for all but finitely many $n$ we have $f(n)\leq g(n)$.
We say that \emph{$f$ is (totally) dominated  by $g$}, denoted $f\leq g$, if for all $n\in\omega$ we have that $f(n)\leq g(n)$.
$D\subseteq\omega^\omega$ is called a \emph{dominating family} if every function in $\omega^\omega$ is dominated by some element of $D$. $\bfrak$, the \emph{(un)bounding number}, is the least size of a subset of $\omega^\omega$ whose elements are not dominated by a single real in $\omega^\omega$. Dually, $\dfrak$, the \emph{dominating number}, is the least size of a dominating family.

For $a,x\in[\omega]^\omega$, we say that \emph{$a$ splits $x$} if both $a\cap x$ and $x\menos a$ are infinite. A subset $S$ of $[\omega]^\omega$ is called a  \emph{splitting family} if any infinite subset of $\omega$ is split by some member of $S$. For $x\in[\omega]^\omega$ and $F\subseteq[\omega]^\omega$, we say that \emph{$x$ reaps $F$} if $x$ splits all elements of $F$. $\sfrak$, the \emph{splitting number}, is defined as the least size of a splitting family. Dually, $\rfrak$, the \emph{reaping number}, is defined as the least size of a subset of $[\omega]^\omega$ that cannot be reaped by a single infinite subset of $\omega$.

A family $A\subseteq[\omega]^\omega$ is said to be \emph{almost disjoint}, abbreviated a.d., if the intersection of any two different members of $A$ is finite. An infinite almost disjoint family is called a \emph{maximal almost disjoint family}, abbreviated mad family, if it is maximal under inclusion among such a.d. families. By $\afrak$ we denote the least size of a mad family and refer to it as the \emph{almost disjointness number}.  Following standard practice, whenever $a,b$ are subset of $\omega$, we denote by $a\subseteq^* b$ the fact that $a\menos b$ is finite. For $C\subseteq[\omega]^{\omega}$ say that $x\in[\omega]^\omega$ is a \emph{pseudo-intersection of $C$} if $x\subseteq^* a$ for any $a\in C$. A family $F\subseteq[\omega]^{\omega}$ is called a \emph{filter base} if it is closed under intersections. The \emph{pseudo-intersection number $\pfrak$} is defined as the least size of a filter base without a pseudo-intersection. The \emph{ultrafilter number $\ufrak$} is defined as the least size of a filter base that generates a non-principal ultrafilter on $\omega$.

A family $\Gwf$ of infinite subsets of $\omega$ is \emph{groupwise-dense} if $\Gwf$ is downward closed under $\subseteq^*$ and, for any interval partition $\langle I_n\rangle_{n<\omega}$ of $\omega$, there exists an $A\in[\omega]^\omega$ such that $\bigcup_{n\in A}I_n\in\Gwf$. The \emph{groupwise-density number} $\gfrak$ is the least size of a family of groupwise-dense sets whose intersection is empty.

For an uncountable Polish space with a continuous\footnote{In the sense that the singletons have measure zero.} Borel probability measure, let $\Mwf$ be the $\sigma$-ideal of meager sets and let $\Nwf$ be the $\sigma$-ideal of null sets. For $\Iwf$ being $\Mwf$ or $\Nwf$, the following cardinal invariants are defined. Note that their values do not depend on the underlying Polish space:
\begin{description}
   \item[$\add(\Iwf)$] \emph{The additivity of $\Iwf$}, which is the least size of a family $F\subseteq\Iwf$ whose union is not in $\Iwf$.
   \item[$\cov(\Iwf)$] \emph{The covering of $\Iwf$}, which is the least size of a family $F\subseteq\Iwf$ whose union covers all the reals.
   \item[$\non(\Iwf)$] \emph{The uniformity of $\Iwf$}, which is the least size of a set of reals not in $\Iwf$.
   \item[$\cof(\Iwf)$] \emph{The cofinality of $\Iwf$}, which is the least size of a cofinal subfamily of $\langle\Iwf,\subseteq\rangle$.
\end{description}
\begin{figure}
\begin{center}
  \includegraphics{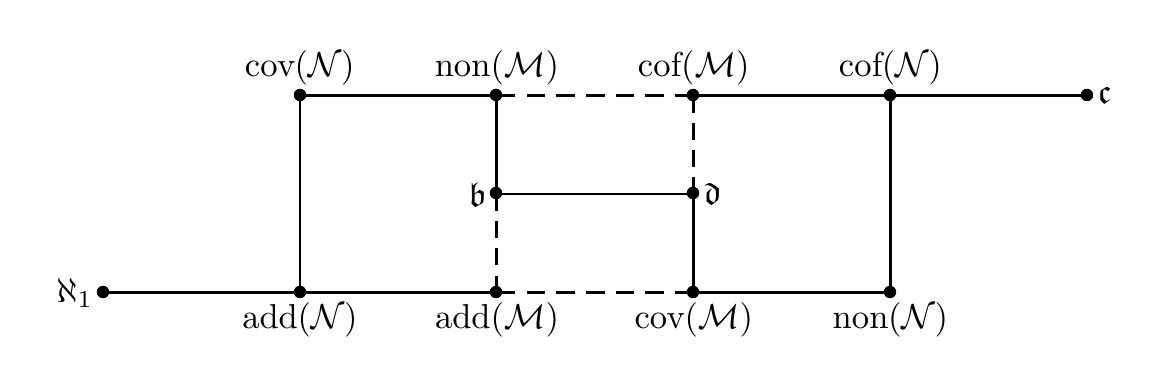}
\caption{Cicho\'n's diagram}
\label{fig:1}
\end{center}
\end{figure}

We will use the following characterizations of $\add(\Nwf)$ and $\cof(\Nwf)$ (see \cite[Thm. 2.3.9]{judabarto}). Recall that a function $\psi:\omega\to[\omega]^{<\omega}$ is called a \emph{slalom}. For $x\in\omega^\omega$ and a slalom $\psi$, we say that \emph{$\psi$ localizes $x$}, denoted $x\in^*\psi$ if for all but finitely many $n$, $x(n)\in\psi(n)$. For a function $h:\omega\to\omega$, denote by $S(\omega,h)$ the set of all slaloms $\psi$ such that $|\psi(n)|\leq h(n)$ for all $n$. If $h(n)$ goes to infinity, then $\add(\Nwf)$ is the least size of a family of reals in $\omega^\omega$ that cannot be localized by a single slalom in $S(\omega,h)$ and, dually, $\cof(\Nwf)$ is the least size of a family of slaloms $S\subseteq S(\omega,h)$ such that any real in $\omega^\omega$ is localized by some slalom in $S$.

The well known Cicho\'n's diagram (Figure \ref{fig:1}) illustrates all provable (in $\mathrm{ZFC}$) inequalities between the relevant cardinal characteristics. The vertical lines from bottom to top and horizontal lines from left to right represent $\leq$. Also, the dotted lines mean $\add(\Mwf)=\min\{\bfrak,\cov(\Mwf)\}$ and $\cof(\Mwf)=\max\{\dfrak,\non(\Mwf)\}$. In addition we have $\pfrak\leq\add(\Mwf)$, $\pfrak\leq\sfrak$, $\pfrak\leq\gfrak$, $\sfrak\leq\dfrak$, $\gfrak\leq\dfrak$, $\bfrak\leq\afrak$, $\bfrak\leq\rfrak$, $\sfrak\leq\non(\Iwf)$, $\cov(\Iwf)\leq\rfrak$ (where $\Iwf$ is $\Mwf$ or $\Nwf$) and $\rfrak\leq\ufrak$. Note that the characteristics $\add(\Nwf)$, $\add(\Mwf)$, $\bfrak$, $\pfrak$ and $\gfrak$ are regular, and that there are no other ZFC provable inequalities between these invariants.

The following results is a very useful tool for consistency results about $\gfrak$.

\begin{lemma}[Blass {\cite[Thm. 2]{Bl}},  see also {\cite[Lemma 1.17]{Br-Suslin}}]\label{lemmagfrak}
   If $\langle W_\alpha\rangle_{\alpha\leq\theta}$ an increasing sequence of transitive models of $\thzfc$ such that
   \begin{enumerate}[(i)]
      \item $[\omega]^\omega\cap(W_{\alpha+1}\menos W_{\alpha})\neq\emptyset$,
      \item $\langle[\omega]^\omega\cap W_\alpha\rangle_{\alpha<\theta}\in W_\theta$ and
      \item $[\omega]^\omega\cap W_\theta=\bigcup_{\alpha<\theta}[\omega]^\omega\cap W_\alpha$.
   \end{enumerate}
   Then, in $W_\theta$, $\gfrak\leq\theta$.
\end{lemma}

\subsection{Forcing theory}\label{SubSecForcing} Excellent references for the theory of forcing are \cite{judabarto}, \cite{jech} and \cite{kunen11}.

Let $\mathbb{P}$  and $\mathbb{Q}$ be partial orders. Then $\mathbb{P}$ is said to be a \emph{subposet}  of $\mathbb{Q}$ if $\Por\subseteq\Qor$ (as partial orders) and incompatibilities are preserved, that is whenever $p\perp_{\Por} q$ (that is, there is no condition in $\Por$ stronger than both $p$ and $q$) then $p\perp_{\Qor}q$. We say that \emph{$\Por$ is a complete suborder, also complete subposet of $\Qor$}, which we denote $\Por\lessdot\Qor$, if $\Por$ is a subposet of $\Qor$ and every maximal antichain of $\mathbb{P}$ is a maximal antichain of $\mathbb{Q}$. If $M$ is a transitive model of $\thzfc$ and $\Por\in M$, then $\Por\lessdot_M\Qor$ denotes the fact that $\Por$ is a subposet of $\Qor$ and every maximal antichain $A$ of $\Por$ which is an element of $M$ is a maximal antichain of $\Qor$.

\begin{definition}[Mathias-Prikry type forcing]\label{DefMatLavUf}
   Let $F$ be a filter subbase. \emph{Mathias-Prikry forcing with $F$} is the poset $\Mor_F$ consisting of all pairs $(s,a)$ such that  $s\in[\omega]^{<\omega}$,  $a\in F$ and  $\sup(s+1)\leq\min(a)$ where $s+1=\{k+1 : k\in s\}$, and ordered by $(t,b)\leq(s,a)$ iff $s\subseteq t$, $b\subseteq a$ and $t\menos s\subseteq a$.
\end{definition}

$\Mor_F$ is $\sigma$-centered. It adds a pseudo-intersection of $F$ which is often referred to as the \emph{Mathias-Prikry real added by $\Mor_{F}$}.

\begin{definition}[Suslin ccc poset]\label{DefSuslinposet}
   A \emph{Suslin ccc poset} $\Sor$ is a ccc poset, whose conditions are reals (in some fixed uncountable Polish space) such that the relations $\leq$ and $\perp$ are $\boldsymbol{\Sigma}_1^1$.
\end{definition}

If $\Sor$ is a Suslin ccc poset then $\Sor$ itself has a $\boldsymbol{\Sigma}_1^1$-definition, because $x\in\Sor$ iff $x\leq x$. Also, if $M\subseteq N$ are transitive models of $\thzfc$ and $\Sor$ is coded in $M$, then $\Sor^M\lessdot_M\Sor^N$.

\begin{definition}[{\cite{brendle05}}]\label{DefSuslinLinked}
   Let $\Sor$ be a Suslin ccc poset.
   \begin{enumerate}[(1)]
      \item $\Sor$ is \emph{Suslin $\sigma$-linked} if there exists a sequence $\{S_n\}_{n<\omega}$ of 2-linked subsets of $\Sor$ such that the statement ``$x\in S_n$" is $\boldsymbol{\Sigma}^1_1$. Note that the statement ``$S_n$ is 2-linked" is $\boldsymbol{\Pi}_1^1$.
      \item $\Sor$ is \emph{Suslin $\sigma$-centered} if there exists a sequence $\{S_n\}_{n<\omega}$ of centered subsets of $\Sor$ such that the statement ``$x\in S_n$" is $\boldsymbol{\Sigma}^1_1$. Note that the statement ``$S_n$ is centered" is $\boldsymbol{\Pi}_2^1$, since the statement ``$p_0,\ldots,p_l$ have a common stronger condition in $\Sor$" is $\boldsymbol{\Sigma}^1_1$.
   \end{enumerate}
\end{definition}

The following are well known examples of Suslin ccc notions, which will be used in our applications. Their order and incompatibility relations are Borel.
\begin{itemize}
   \item \emph{Cohen forcing $\Cor$}.
   \item \emph{Random forcing $\Bor$}.
   \item \emph{Hechler forcing $\Dor$}, the canonical ccc forcing that adds a dominating real.
   \item Let $h:\omega\to\omega$ non-decreasing and converging to infinity. $\Loc^h$, the \emph{localization forcing at $h$}, consists of conditions of the form $(s,F)$ where $s\in\prod_{i<n}[\omega]^{\leq h(i)}$ and $F\in[\omega^\omega]^{\leq h(n)}$ for some $n<\omega$. The order is $(s',F')\leq(s,F)$ iff $s\subseteq s'$, $F\subseteq F'$ and $\{x(i) : x\in F\}\subseteq s'(i)$ for all $i\in|s'|\menos|s|$. $\Loc:=\Loc^{id}$ where $id:\omega\to\omega$ is the identity function.
\end{itemize}
 Moreover $\Cor$ and $\Dor$ are Suslin $\sigma$-centered, while $\Loc^h$ and $\Bor$ are Suslin $\sigma$-linked. For each of these posets the statement ``$p_0,\ldots,p_l$ have a common stronger condition" is Borel. Then for any $\boldsymbol{\Sigma}^1_1$-subset $S$ of such a poset, the statement ``$S$ is centered" is $\boldsymbol{\Pi}_1^1$.

The notion of correctness, which we state below and which is introduced by Brendle in~\cite{brendle05}, is essential for the construction of template
iterations.

\begin{definition}[Correct diagram of posets {\cite[Def. 1.1]{brendle05}}]\label{DefCorr}
   For $i=0,1$, let $\Por_i$ and $\Qor_i$ be posets. If $\Por_i\lessdot\Qor_i$ for $i=0,1$, $\Por_0\lessdot\Por_1$ and $\Qor_0\lessdot\Qor_1$, say that the diagram $\langle\Por_0,\Por_1,\Qor_0,\Qor_1\rangle$ (see Figure \ref{fig:2})  is \emph{correct} if for each $q\in\Qor_0$ and $p\in\Por_1$, if they have a common reduction in $\Por_0$, then they are compatible in $\Qor_1$. An equivalent formulation is that, whenever $p_0\in\Por_0$ is a reduction of $p_1\in\Por_1$, then $p_0$ is a reduction of $p_1$ with respect to $\Qor_0,\Qor_1$.
   \begin{figure}
     \begin{center}
         \includegraphics{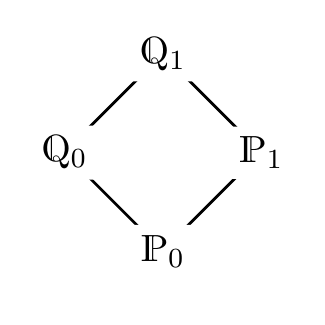}
         \caption{Diagram of posets}
         \label{fig:2}
     \end{center}
   \end{figure}
\end{definition}

\begin{definition}[{\cite{brendle05}}]\label{DefSuslincorr}
 A Suslin ccc poset $\Sor$ is \emph{correctness-preserving} if, given a correct diagram $\langle\Por_0,\Por_1,\Qor_0,\Qor_1\rangle$, the diagram $\langle\Por_0\ast\Snm^{V^{\Por_0}},\Por_1\ast\Snm^{V^{\Por_1}},\Qor_0\ast\Snm^{V^{\Qor_0}},
 \Qor_1\ast\Snm^{V^{\Qor_1}}\rangle$ is also correct.
\end{definition}

Brendle showed that all Suslin ccc posets listed above are correctness-preserving (see~\cite{brendle05,Br-shat}). In addition, he conjectured that any Suslin ccc poset is correctness-preserving, which is still open.


\section{Template iterations}\label{SecTemp}

This section contains definitions of template and template iterations, as well as a discussion of some of their basic properties. The exposition follows
\cite{mejia-temp}. Proofs of all statements can be found in \cite{mejia-temp}, \cite{brendle02,brendle05}.


For a linear order $L:=\langle L,\leq_L\rangle$ and $x\in L$, denote $L_x:=\left\{z\in L : z<x\right\}$.

\begin{definition}[Indexed template]\label{DefIndxTemp}
An \emph{indexed template} (or just a \emph{template}) is a pair $\langle L,\bar{\Iwf}:=\langle\Iwf_x\rangle_{x\in L}\rangle$ where $L$ is a linear order, $\Iwf_x\subseteq\Pwf(L_x)$ for all $x\in L$ and the following properties are satisfied:
\begin{enumerate}[(1)]
     \item $\emptyset\in\Iwf_x$.
     \item $\Iwf_x$ is closed under finite unions and intersections.
     \item If $z<x$ then there is some $A\in\Iwf_x$ such that $z\in A$.
     \item $\Iwf_x\subseteq\Iwf_y$ if $x<y$.
     \item $\Iwf(L):=\bigcup_{x\in L}\Iwf_x\cup\{L\}$ is well-founded by the subset relation.
  \end{enumerate}
   For $A\subseteq L$ and $x\in L$, $\Iwf_x\frestr A:=\left\{A\cap X : X\in\Iwf_x\right\}$ is the \emph{trace of $\Iwf_x$ on $A$}. Let $\bar{\Iwf}\frestr A:=\langle\Iwf_x\frestr A\rangle_{x\in A}$ and\footnote{$\Iwf(A)\subseteq\Iwf(L)\frestr A=\{A\cap X : X\in\Iwf(L)\}$ but equality may not hold.} $\Iwf(A):=\bigcup_{x\in A}\Iwf_x\frestr A\cup\{A\}$.
\end{definition}

If $X\subseteq A\subseteq L$, then $(\Iwf_x\frestr A)\frestr X=\Iwf_x\frestr X$ for any $x\in L$; $(\bar{\Iwf}\frestr A)\frestr X=\bar{\Iwf}\frestr X$ and $(\Iwf(A))(X)=\Iwf(X)$. As $\langle A,\bar{\Iwf}\frestr A\rangle$ is an indexed template for any $A\subseteq L$, we can define $\Dp^{\bar{\Iwf}}:\PO(L)\to\mathbf{ON}$ by $\Dp^{\bar{\Iwf}}(X):=\rank_{\Iwf(X)}(X)$. Although this is not a rank function on $\PO(L)$, we will use induction on $\alpha=\Dp^{\bar{\Iwf}}(X)$ to construct an iteration along $\langle L,\bar{\Iwf}\rangle$. When the template is clear from the context, we just denote $\mathrm{Dp}:=\mathrm{Dp}^{\bar{\Iwf}}$.

\begin{lemma}[{\cite[Lemma 3.3]{mejia-temp}}]\label{UpsilonTemp}
   Fix $A\subseteq L$. $\mathrm{Dp}:=\mathrm{Dp}^{\bar{\Iwf}}$ has the following properties.
   \begin{enumerate}[(a)]
      \item If $Y\in\Iwf(A)$, then $\mathrm{Dp}(Y)\leq\mathrm{rank}_{\Iwf(A)}(Y)$.
      \item If $X\subseteq A$ then $\mathrm{Dp}(X)\leq\mathrm{Dp}(A)$.
      \item Let $x\in A$. If $Y\subsetneq A\cap(L_x\cup\{x\})$ and $Y\cap L_x\in\Iwf_x\frestr A$ then $\mathrm{Dp}(Y)<\mathrm{Dp}(A)$. In particular, $\mathrm{Dp}(X)<\mathrm{Dp}(A)$ for all $X\in\Iwf_x\frestr A$
      \item $\mathrm{Dp}^{\bar{\Iwf}\upharpoonright A}=\mathrm{Dp}\frestr\PO(A)$.
   \end{enumerate}
\end{lemma}

Given an indexed template $\langle L,\bar{\Iwf}\rangle$ and $x\in L$, define $\hat{\Iwf}_x=\{B\subseteq L :(\exists{H\in\Iwf_x})(B\subseteq H)\}$. Thus $\hat{\Iwf}_x$ is the ideal on $\PO(L_x)$ generated by $\Iwf_x$ (which might be trivial). Note that $\hat{\Iwf}_x$ contains all the finite subsets of $L_x$ and that $B\in\hat{\Iwf}_x$ iff $B\in\Iwf_x\frestr(B\cup\{x\})$ for any $B\subseteq L_x$.

\begin{theorem}[Iteration along a template]\label{TempIt}
   Given a template $\langle L,\bar{\Iwf}\rangle$, a partial order $\Por\frestr A$ is defined by recursion on $\alpha=\mathrm{Dp}(A)$ for all $A\subseteq L$ as follows:
   \begin{enumerate}[(1)]
      \item For $x\in L$ and $B\in\hat{\Iwf}_x$, $\Qnm^B_x$ is a $\Por\frestr B$-name of a poset. The following conditions should hold.
            \begin{enumerate}[(i)]
               \item If $E\subseteq B$ and $\Por\frestr E\lessdot\Por\frestr B$, then $\Vdash_{\Por\upharpoonright B}\Qnm_x^E\lessdot_{V^{\Por\upharpoonright E}}\Qnm_x^B$.
               \item If $E\in\hat{\Iwf}_x$ such that $\Por\frestr(B\cap E)$ is a complete subposet of both $\Por\frestr B$ and $\Por\frestr E$ and if $\dot{q}$ is a $\Por\frestr(B\cap E)$-name such that $\Vdash_{\Por\upharpoonright E}\dot{q}\in\Qnm_x^E$                     and $\Vdash_{\Por\upharpoonright B}\dot{q}\in\Qnm_x^B$, then $\Vdash_{\Por\upharpoonright(B\cap E)}\dot{q}\in\Qnm_x^{B\cap E}$.
               \item If $B',D\subseteq B$ and $\langle\Por\frestr(B'\cap D),\Por\frestr B',\Por\frestr D,\Por\frestr B\rangle$ is a correct diagram, then the diagram ${\langle\Por\frestr(B'\cap D)\ast\Qnm_x^{B'\cap D}},\Por\frestr B'\ast\Qnm_x^{B'},\Por\frestr D\ast\Qnm_x^D,\Por\frestr B\ast\Qnm_x^B\rangle$ is correct.
            \end{enumerate}
      \item The partial order $\Por\frestr A$ is defined as follows.
            \begin{enumerate}[(i)]
               \item $\Por\frestr A$ consists of all finite partial functions $p$ with domain contained in $A$ such that $p=\emptyset$ or if $|p|>0$ and $x=\max(\dom p)$, then there exists a $B\in\Iwf_x\frestr A$ such that $p\frestr L_x\in\Por\frestr B$ and $p(x)$ is a $\Por\frestr B$-name for a condition in $\Qnm_x^B$.
               \item The order on $\Por\frestr A$ is given by: $q\leq_A p$ if $\dom p\subseteq\dom q$ and either $p=\emptyset$ or when $p\neq0$ and $x=\max(\dom q)$ then there is a $B\in\Iwf_x\frestr A$ such that $q\frestr L_x\in\Por\frestr B$ and, either $x\notin\dom p$, $p\in\Por\frestr B$ and $q\frestr L_x\leq_B p$, or $x\in\dom p$, $p\frestr L_x\in\Por\frestr B$, $q\frestr L_x\leq_B p\frestr L_x$ and $p(x),q(x)$ are $\Por\frestr B$-names for conditions in $\Qnm_x^B$ such that $q\frestr L_x\Vdash_{\Por\upharpoonright B}q(x)\leq p(x)$.
            \end{enumerate}
   \end{enumerate}
   Within this recursive definition, the following properties are proved:
   \begin{enumerate}[(a)]
      \item If $p\in\Por\frestr A$, $x\in A$ and $\max(\dom p)<x$, then there exists $B\in\Iwf_x\frestr A$ such that $p\in\Por\frestr B$.
      \item For $D\subseteq A$, $\Por\frestr D\subseteq\Por\frestr A$ and for $p,q\in\Por\frestr D$, $q\leq_D p$ iff $q\leq_A p$.
      \item $\Por\frestr A$ is a poset.
      \item $\Por\frestr A$ is obtained from posets of the form $\Por\frestr B$ with $B\subsetneq A$ in the following way:
            \begin{enumerate}[(i)]
               \item If $x=\max(A)$ exists and $A_x:=A\cap L_x\in\hat{\Iwf}_x$, then $\Por\frestr A=\Por\frestr A_x\ast\Qnm_x^{A_x}$.
               \item If $x=\max(A)$ but $A_x\notin\hat{\Iwf}_x$, then $\Por\frestr A$ is the direct limit of the $\Por\frestr B$ where
                     $B\subseteq A$ and $B\cap L_x\in\Iwf_x\frestr A$.
               \item If $A$ does not have a maximum element, then $\Por\frestr A$ is the direct limit of the $\Por\frestr B$ where $B\in\Iwf_x\frestr A$ for some $x\in A$ (in the case $A=\emptyset$, it is clear that $\Por\frestr A=\mathds{1}$).
            \end{enumerate}
            Note that by Lemma \ref{UpsilonTemp}(c) we have $\mathrm{Dp}(A_x)<\mathrm{Dp}(A)$ in (i) and, in (ii) and (iii) we have $\mathrm{Dp}(B)<\mathrm{Dp}(A)$ for each corresponding $B$.
      \item If $D\subseteq A$, then $\Por\frestr D\lessdot\Por\frestr A$.
      \item If $D\subseteq L$ then $\Por\frestr(A\cap D)=\Por\frestr A\cap\Por\frestr D$.
      \item If $D,A'\subseteq A$ then $\langle\Por\frestr(A'\cap D),\Por\frestr A',\Por\frestr D,\Por\frestr A\rangle$ is a correct diagram.
   \end{enumerate}
\end{theorem}
\begin{proof}
   See \cite[Thm. 2.2]{brendle05} or \cite[Thm 4.1]{mejia-temp}.
\end{proof}

We are particularly interested in ccc template iterations.

\begin{lemma}\label{templateitccc}
   Let $\Por\frestr\langle L,\bar{\Iwf}\rangle$ be a template iteration such that:
   \begin{enumerate}[(i)]
      \item for all $x\in L$, $B\in\hat{\Iwf}_x$ there are $\Por\frestr B$-names $\langle\dot{Q}^B_{x,n}\rangle_{n<\omega}$ which
            witness that $\Qnm^B_x$ is $\sigma$-linked;
      \item if $D\subseteq B$ then $\Vdash_{\Por\upharpoonright B}\dot{Q}^D_{x,n}\subseteq\dot{Q}^B_{x,n}$ for all $n<\omega$.
   \end{enumerate}
   Then $\Por\frestr L$ has the Knaster condition.
\end{lemma}
\begin{proof}
   See \cite[Lemma 2.3]{brendle05} and \cite[Lemma 4.5]{mejia-temp}.
\end{proof}

In our applications, we will be using template iterations of the following form:

\begin{definition}\label{DefStandardIt}
   Let $\langle L,\bar{\Iwf}\rangle$ be an indexed template. An iteration $\Por\frestr\langle L,\bar{\Iwf}\rangle$ is \emph{standard} if
   \begin{enumerate}[(i)]
      \item $L=L_S\cup L_C$ is a disjoint union,
      \item for $x\in L_S$, $\mathbb{S}_x$ is a fixed Suslin $\sigma$-linked correctness-preserving forcing notion coded in the ground model,
      \item for $x\in L_S$ and $B\in\hat{\Iwf}_x$, $\Qnm_x^B$ is a $\Por\frestr B$-name for $\mathbb{S}_x^{V^{\Por\upharpoonright B}}$,
      \item for $x\in L_C$, $C_x$ is a fixed set in $\hat{\Iwf}_x$, $\Qnm_x$ is a $\Por\frestr C_x$-name for a $\sigma$-linked poset whose conditions are reals\footnote{These reals belong to some fixed uncountable Polish space $R_x$ coded in the ground model.}, and
      \item for $x\in L_C$ and $B\in\hat{\Iwf}_x$ the name $\Qnm_x^B$ is either $\Qnm_x$ in case $C_x\subseteq B$, or it is a name for the trivial poset otherwise.
   \end{enumerate}
   If $\theta$ is a cardinal, say that the iteration is \emph{$\theta$-standard} if, additionally, $|C_x|<\theta$ for all $x\in L_C$.
\end{definition}


\begin{lemma}\label{CondSupp}
  Let $\theta$ be a cardinal with uncountable cofinality and let $\Por\frestr\langle L,\bar{\Iwf}\rangle$ be a $\theta$-standard template iteration. Then for each $A\subseteq L$,
  \begin{enumerate}[(a)]
     \item $\Por\frestr A$ is Knaster,
     \item if $p\in\Por\frestr A$ then there is $C\subseteq A$ of size $<\theta$ such that $p\in\Por\frestr C$, and
     \item if $\dot{x}$ is a $\Por\frestr A$-name for a real, then there is $C'\subseteq A$ of size $<\theta$ such that $\dot{x}$ is a $\Por\frestr C'$-name.
  \end{enumerate}
\end{lemma}
\begin{proof}
   See \cite[Lemma 2.4]{brendle05} and \cite[Lemma 4.6]{mejia-temp}.
\end{proof}

We will use Shelah's notion of innocuous extension to give a sufficient condition for the forcing equivalence of two distinct standard template iterations.

\begin{definition}[Innocuous extension]\label{DefInnoc}
   Let $\langle L,\bar{\Iwf}\rangle$ be an indexed template and $\theta$ an uncountable cardinal. An indexed template $\langle L,\bar{\Jwf}\rangle$ is a \emph{$\theta$-innocuous extension of $\langle L,\bar{\Iwf}\rangle$} if
           \begin{enumerate}[(i)]
              \item for every $x\in L$, $\Iwf_x\subseteq\Jwf_x$ and
              \item for any $x\in L$ and $X\in\hat{\Jwf}_x$, if $|X|<\theta$ then $X\in\hat{\Iwf}_x$.
           \end{enumerate}
\end{definition}

\begin{definition}\label{DefTempIsom}
   Let $\langle L,\bar{\Iwf}\rangle$ and $\langle L^*,\bar{\Iwf}^*\rangle$ be templates. A function $h:\langle L^*,\bar{\Iwf}^*\rangle\to\langle L,\bar{\Iwf}\rangle$ is a \emph{template-isomorphism} iff it is a bijection that satisfies for all $x,y\in L^*$:
   \begin{enumerate}[(i)]
      \item $x<y$ iff $h(x)<h(y)$ and
      \item $\Iwf_{h(x)}=\{h[A]:A\in\Iwf^*_x\}$.
   \end{enumerate}
\end{definition}

\begin{lemma}\label{InnEqv}
   Let $\theta$ be a cardinal with uncountable cofinality, $\langle L,\bar{\Iwf}\rangle$ and $\langle L^*,\bar{\Iwf}^*\rangle$ templates and $h:\langle L^*,\bar{\Iwf}^*\rangle\to\langle L,\bar{\Iwf}\rangle$ a template-isomorphism.
   Let $\langle L^*,\bar{\Jwf}\rangle$ be a $\theta$-innocuous extension of $\langle L^*,\bar{\Iwf}^*\rangle$. Let $\Por\frestr\langle L,\bar{\Iwf}\rangle$ and $\Por^*\frestr\langle L^*,\bar{\Jwf}\rangle$ be $\theta$-standard template iterations such that:
   \begin{enumerate}[(1)]
      \item $h[L^*_S]=L_S$ and $h[L^*_C]=L_C$;
      \item for $y\in L^*_S$, $\mathbb{S}^*_y=\mathbb{S}_{h(y)}$;
      \item if $y\in L^*_C$ then $h[C^*_y]=C_{h(y)}$ and, whenever there is a sequence $\langle\hat{h}_D : D\subseteq C^*_y\rangle$ of functions such that
      \begin{enumerate}[(3.1)]
      \setcounter{enumii}{-1}
       \item $\hat{h}_D:\Por^*\frestr D\to\Por\frestr h[D]$ is an isomorphism,
       \item $Y\subseteq D$ implies $\hat{h}_Y\subseteq\hat{h}_D$,
       \item for $z\in D\cap L^*_C$ and $E\in\Pwf(D)\cap\hat{\Jwf}_z$, $\Qnm^{h[E]}_{h(z)}$ is the name associated to $\Qnm_z^{*E}$ via $\hat{h}_E$ and,
       \item for $p\in\Por^*\frestr D$, $\dom(\hat{h}_D(p))=h[\dom p]$ and, if $z=\max(\dom p)$, $E\in\Jwf_z\frestr D$, $p\frestr L^*_z\in\Por^*\frestr E$ and $p(z)$ is a $\Por^*\frestr E$-name for a member of $\Qnm_z^{*E}$, then $\hat{h}_D(p)\frestr L_{h(z)}=\hat{h}_E(p\frestr L^*_z)$ and $\hat{h}_D(p)(h(z))$ is the $\Por\frestr h[E]$-name associated to $p(z)$ via $\hat{h}_E$,
      \end{enumerate}
      then $\Qnm_{h(y)}$ is the name associated to $\Qnm^*_{y}$ via $\hat{h}_{C^*_y}$.
   \end{enumerate}
   Then, there exists a unique sequence $\langle\hat{h}_D:D\in[L^*]^{<\theta}\rangle$ satisfying (3.0)-(3.3). Moreover, $\hat{h}:=\bigcup\{\hat{h}_D:D\in[L^*]^{<\theta}\}$ is an isomorphism from $\Por^*\frestr L^*$ onto $\Por\frestr L$ and, for any $Y\subseteq L^*$, $\hat{h}\frestr(\Por^*\frestr Y)=\bigcup\{\hat{h}_D:D\in[Y]^{<\theta}\}$ is an isomorphism onto $\Por\frestr h[Y]$.
\end{lemma}

\begin{remark}\label{RemInnEqv}
   The previous lemma is a more detailed version of \cite[Lemma 1.7]{brendle02} and \cite[Lemma 4.8]{mejia-temp} that we present for constructive purposes. Note that, whenever $z\in L^*$ and $E\in[L^*_z]^{<\theta}$, then $E\in\hat{\Jwf}_z$ iff $h[E]\in\hat{\Iwf}_{h(z)}$, this because $\langle L^*,\bar{\Jwf}\rangle$ is a $\theta$-innocuous extension of $\langle L^*,\bar{\Iwf}^*\rangle$ and by properties (i) and (ii). For this reason, (3.2) makes sense as $\Qnm^{h[E]}_{h(z)}$ is defined iff $\Qnm_z^{*E}$ is. Moreover, the lemma directly implies that the sequence in (3) exists and is unique for each $C^*_y$.

   However, properties (3.0)-(3.3) are restricted to subsets $D$ of size $<\theta$ because there may be an $E\in\hat{\Jwf}_z$ of size bigger than or equal to $\theta$ such that $h[E]\notin\hat{\Iwf}_{h(z)}$, so $\Qnm^{h[E]}_{h(z)}$ is undefined. When $\hat{\Jwf}_z=\hat{\Iwf}^*_z$ we don't have that problem.
\end{remark}

\begin{corollary}\label{IsomTemp}
   With the same hypotheses as in Lemma \ref{InnEqv}, assume further that $\hat{\Jwf}_z=\hat{\Iwf}^*_z$ for all $z\in L^*$. Then there is a unique sequence $\langle\hat{h}_Y:Y\subseteq L^*\rangle$ satisfying (3.0)-(3.3). Moreover, $\hat{h}_Y=\hat{h}_{L^*}\frestr(\Por^*\frestr Y)$ for any $Y\subseteq L^*$.
\end{corollary}

\begin{proof}[Proof of Lemma \ref{InnEqv}]
   We construct $\hat{h}_D$ by induction on $\mathrm{Dp}^{\bar{\Jwf}}(D)$ for $D\in[L^*]^{<\theta}$. Let $p\in\Por^*\frestr D$. If $\dom p=\emptyset$ then $\hat{h}_D(p)$ is the empty sequence, so assume that $\dom p$ is non-empty with maximum $z$. By Theorem \ref{TempIt}(2) there is $E\in\Jwf_z\frestr D$ such that $p\frestr L^*_z\in\Por^*\frestr E$ and $p(z)$ is a $\Por^*\frestr E$-name for a condition in $\Qnm^{*E}_z$. By induction hypothesis, we know $\hat{h}_E$. We split into cases to show that $\Qnm^{h[E]}_{h(z)}$ is the $\Por\frestr h[E]$-name associated to $\Qnm^{*E}_{z}$ via $\hat{h}_E$.
   \begin{itemize}
      \item \emph{$z\in L^*_S$}. By (1) $h(z)\in L_S$ and, by (2), $\Qnm^{h[E]}_{h(z)}$ is a name for $\Sor_{h(z)}^{V^{\Por\rest h[E]}}=\Sor_z^{*V^{\Por^*\rest E}}$.
      \item \emph{$z\in L^*_C$ and $C^*_z\nsubseteq E$}. Then, $C_{h(z)}\nsubseteq h[E]$ and both $\Qnm^{*E}_z$ and $\Qnm^{h[E]}_{h(z)}$ are names for the trivial poset.
      \item \emph{$z\in L^*_C$ and $C^*_z\subseteq E$}. Then, $C_{h(z)}\subseteq h[E]$ and, by induction hypothesis, $\Qnm_{h(z)}$ is the name associated to $\Qnm^*_z$ via $\hat{h}_{C^*_z}$, so $\Qnm^{h[E]}_{h(z)}=\Qnm_{h(z)}$ is the name associated to $\Qnm^{*E}_z=\Qnm^*_z$ via $\hat{h}_{E}$ (because $\hat{h}_{C^*_z}\subseteq\hat{h}_E$).
   \end{itemize}
   Let $\dot{r}$ be the $\Por\frestr h[E]$-name associated to $p(z)$ via $\hat{h}_E$, which is indeed a name for a condition in $\Qnm^{h[E]}_{h(z)}$. Put $\hat{h}_D(p)=\hat{h}_E(p\frestr L^*_z)\cup\{(h(z),\dot{r})\}$, which is a condition in $\Por\frestr h[D]$ ($h[E]\in\hat{\Iwf}_{h(z)}$ by Remark \ref{RemInnEqv} but, in spite that it may not be in $\Iwf_{h(z)}\frestr h[D]$, we can find a $B\in\Iwf_{h(z)}\frestr h[D]$ containing $h[E]$ so $\hat{h}_E(p\frestr L^*_z)\in\Por\frestr B$ and $\dot{r}$ is a $\Por\frestr B$-name of a condition in $\Qnm^B_{h(z)}$). Note that $\hat{h}_D(p)$ does not depend on the chosen $E$ because, if we use some other $E'\in\Jwf_z\frestr D$, then $E''=E\cup E'\in\Jwf_z\frestr D$ and $\hat{h}_{E''}$ extends both $\hat{h}_E$ and $\hat{h}_{E'}$ by induction hypothesis, so $\dot{r}$ is the same name via any of those three functions and $\hat{h}_E(p\frestr L^*_z)=\hat{h}_{E''}(p\frestr L^*_z)=\hat{h}_{E'}(p\frestr L^*_z)$. (3.0)-(3.3) are easily verified for $\hat{h}_D$.

   To see uniqueness, let $\langle\hat{h}'_D:D\in[L_z]^{<\theta}\rangle$ be another sequence satisfying (3.0)-(3.3). By (3.3), $\hat{h}'_D=\hat{h}_D$ is easily verified by induction on $\mathrm{Dp}^{\bar{\Jwf}}(D)$ for $D\in[L^*]^{<\theta}$.

   Now let $Y\subseteq L^*$ be arbitrary. Lemma \ref{CondSupp} implies that $\Por^*\frestr Y=\bigcup\{\Por^*\frestr D:D\in[Y]^{<\theta}\}$ and likewise for $\Por\frestr h[Y]$, so $\bigcup\{\hat{h}_D:D\in[Y]^{<\theta}\}$ defines an isomorphism from $\Por^*\frestr Y$ onto $\Por\frestr h[Y]$.
\end{proof}

\begin{lemma}\label{InnConstr}
   Let $\theta$ be a cardinal with uncountable cofinality, $\langle L,\bar{\Iwf}\rangle$ and $\langle L^*,\bar{\Iwf}^*\rangle$ templates and $h:\langle L^*,\bar{\Iwf}^*\rangle\to\langle L,\bar{\Iwf}\rangle$ a template-isomorphism. Let $\Por\frestr\langle L,\bar{\Iwf}\rangle$ be a $\theta$-standard iteration. If $\langle L^*,\bar{\Jwf}\rangle$ is a $\theta$-innocuous extension of $\langle L^*,\bar{\Iwf}^*\rangle$, then there is a $\theta$-standard iteration $\Por^*\langle L^*,\bar{\Jwf}\rangle$ that satisfies (1)-(3) of Lemma \ref{InnEqv}.
\end{lemma}
\begin{proof}
   Define $L^*_S=h^{-1}[L_S]$, $L^*_C=h^{-1}[L_C]$, $\Sor_y=\Sor^*_{h(y)}$ for each $y\in L^*_S$ and $C^*_y=h^{-1}[C_{h(y)}]$ for each $y\in L^*_C$, which is in $\hat{\Iwf}^*_y$ because $\langle L^*,\bar{\Jwf}\rangle$ is a $\theta$-innocuous extension of $\langle L^*,\bar{\Iwf}^*\rangle$ (see Remark \ref{RemInnEqv}). For a fixed $y\in L^*_C$, define $\langle\hat{h}_D:D\subseteq C_y\rangle$ and $\Por^*\frestr D$ satisfying (3.0)-(3.3) by recursion on $\mathrm{Dp}^{\bar{\Jwf}}(D)$. The uniqueness of this sequence can be proved by induction on $\mathrm{Dp}^{\bar{\Jwf}}(D)$, which implies that $\Qnm^*_y$ is well-defined as the $\Por^*\frestr C_y$-name associated to $\Qnm_{h(y)}$ via $\hat{h}_{C_y}$. By Theorem \ref{TempIt}, this is enough to know how to define a standard iteration $\Por^*\langle L^*,\bar{\Jwf}\rangle$ as in Definition \ref{DefStandardIt} that satisfies the desired requirements.
\end{proof}



\section{Shelah's template}\label{SecSelahTemp}

In order to obtain our main result, we introduce a minor modification to the template that Shelah used to prove the consistency of $\mathfrak{d}<\mathfrak{a}$ (without the use of a measurable). Our presentation is based on \cite[Sect. 3]{brendle02}.

Given an ordinal $\alpha$, let $\alpha^*$ denote a disjoint copy of $\alpha$ with a linear order isomorphic to the inverse order of $\alpha$. Let $\mathbf{ON}^*=\{\alpha^*: \alpha\in\mathbf{ON}\}$ where $\mathbf{ON}$ is the class of all ordinals. Members of $\mathbf{ON}$ are called \emph{positive}, while members of $\mathbf{ON}^*$ are \emph{negative}. We order $\mathbf{ON}\cup\mathbf{ON}^*$ in the natural way (like the integers but without a neutral member as $0$ is positive and $0^*$ is negative). For $\xi\in\mathbf{ON}\cup\mathbf{ON}^*$, $\xi+1$ denotes the \emph{immediate successor of $\xi$} and $\xi-1$ the \emph{immediate predecessor of $\xi$}. Note that $0^*+1=0$, $0-1=0^*$, $\xi+1$ does not exists iff $\xi=\gamma^*$ for some limit ordinal $\gamma$, and $\xi-1$ does not exists iff $\xi$ is a limit ordinal (positive).

\begin{definition}\label{DefShelahTemp}
\begin{enumerate}[(1)]
\item  Define $\mathbf{SO}$ as the class of non-empty finite sequences $x$ where $x(0)$ is an ordinal and $x(k)\in\mathbf{ON}\cup\mathbf{ON}^*$ for all $0<k<|x|$. Order $\mathbf{SO}$ as $x<y$ iff either
   \begin{enumerate}[(i)]
      \item there is a $k<\min\{|x|,|y|\}$ such that $x\frestr k=y\frestr k$ and $x(k)<y(k)$,
      \item $x\subseteq y$ and $y(|x|)$ is positive, or
      \item $y\subseteq x$ and $x(|y|)$ is negative.
   \end{enumerate}
   Note that $<$ is a linear order on $\mathbf{SO}$ and that $\mathbf{ON}$, with the canonical well-order, is embedded there. Therefore, we identify the ordinals with the sequences of length 1 in $\mathbf{SO}$.

\item Say that $A\subseteq\mathbf{SO}$ is a \emph{tree} if, whenever $t\in A$ and $t$ end-extends a sequence $s$, then $s\in A$.

\item For non-zero ordinals $\gamma$ and $\delta$ define the set
   \[L^{\delta,\gamma}=\{x\in\mathbf{SO}:x(0)<\gamma\textrm{\ and }\delta^*<x(k)<\delta\textrm{\ for all $0<k<|x|$}\}\]
   linearly ordered by $<$ (the order from $\mathbf{SO}$). Here, $\gamma$ is the \emph{length of $L^{\delta,\gamma}$}, while $\delta$ is its \emph{width}. As before, the members of $\gamma$ are identified with the sequences of length 1 in $L^{\delta,\gamma}$. Clearly, $L^{\delta,\gamma}$ is a tree.

\item Let $\Sigma=\langle S_\beta : \beta<\tau\rangle$ be a partition of $\delta^*$ where $\tau$ is an ordinal and let $\beta^\Sigma:\delta^*\to\tau$ be defined by $\beta^\Sigma(\xi)=\beta$ when $\xi\in S_\beta$. Say that $x\in L^{\delta,\gamma}$ is \emph{$\Sigma$-relevant} iff the following hold:
   \begin{enumerate}[(i)]
      \item $|x|\geq3$ is odd;
      \item for $i<|x|$, $x(i)$ is positive iff $i$ is even;
      \item the sequence $\{\beta^\Sigma(x(i-1))\}_{i\in r_x}$ is decreasing, where $r_x:=\{i<|x| : i\geq2\textrm{\ is even},\ x(i)<\tau\}$ and
      \item $|x|-1\in r_x$.
   \end{enumerate}

   For $\Sigma$-relevant $x\in L^{\delta,\gamma}$ let $J^{\Sigma,\gamma}_x:=\{z\in L^{\delta,\gamma}: x\frestr(|x|-1)\leq z <x\}$. Define $\Iwf^{\Sigma,\gamma}$ as the family of finite unions of the following \emph{basic sets}:
   \begin{itemize}
     \item $L^{\delta,\gamma}_\alpha$ (the segment of objects $<\alpha=\la\alpha\ra$) where $\alpha\in\gamma+1$ (for $\alpha=\gamma$ it is $L^{\delta,\gamma}$).
     \item $J^{\Sigma,\gamma}_x$ where $x\in L^{\delta,\gamma}$ is $\Sigma$-relevant.
     \item $\{z\}$ where $z\in L^{\delta,\gamma}$.
   \end{itemize}
   For $x\in L^{\delta,\gamma}$, put $\Iwf^{\Sigma,\gamma}_x:=\{A\subseteq L^{\delta,\gamma}_x : A\in\Iwf^{\Sigma,\gamma}\}$ and $\bar{\Iwf}^{\Sigma,\gamma}=\langle\Iwf_x^{\Sigma,\gamma}\rangle_{x\in L^{\delta,\gamma}}$.
\end{enumerate}
\end{definition}

Note that any basic set is convex in $L^{\delta,\gamma}$ and that any member of $\Iwf^{\Sigma,\gamma}$ can be written as a disjoint union of basic sets and this disjoint union is unique. This is because, for any two basic sets, either one is contained in the other, or they are disjoint in which case their union is not convex and, thus, not a basic set. Moreover $$\Iwf^{\Sigma,\gamma}=\Iwf^{\Sigma,\gamma}(L^{\delta,\gamma})=\bigcup_{x\in L^{\delta,\gamma}}\Iwf^{\Sigma,\gamma}_x\cup\{ L^{\delta,\gamma}\}.$$

\begin{lemma}\label{ShelahTemp}
   $\langle L^{\delta,\gamma},\bar{\Iwf}^{\Sigma,\gamma}\rangle$ is an indexed template.
\end{lemma}
\begin{proof}
   See \cite[Lemma 3.2]{brendle02}.
\end{proof}

\begin{definition}\label{DefSigmaIt}
   Let $\theta$ be an uncountable regular cardinal and $\mathcal{S}=\langle\mathbb{S}_\eta\rangle_{\eta<\nu}$ be a sequence of Suslin $\sigma$-linked correctness-preserving forcing notions coded in the ground model where $\nu\leq\theta$. A \emph{$(\mathcal{S},\theta)$-standard iteration along a template $\langle L,\bar{\Iwf}\rangle$} is a $\theta$-standard iteration $\Por\frestr\langle L,\bar{\Iwf}\rangle$ (see Definition \ref{DefStandardIt}) where
   \begin{enumerate}[(i)]
      \item $\langle L_{S,\eta}\rangle_{\eta<\nu}$ is a partition of $L_S$,
      \item for $x\in L_{S,\eta}$, $\Sor_x=\Sor_\eta$ and
      \item for $x\in L_C$, $\Qnm_x$ is forced by $\Por\frestr C_x$ to have size $<\theta$. By ccc-ness, without loss of generality we can even say that the domain of $\Qnm_x$ is an ordinal $\gamma_x<\theta$ (in the ground model, not just a name).
   \end{enumerate}
\end{definition}

Until the end of this section, fix $\theta$ and $\mathcal{S}$ as above, $\gamma$ and $\delta$ non-zero ordinals, a partition $\Sigma=\langle S_\beta : \beta<\theta\rangle$ of $\delta^*$,  $L=L^{\delta,\gamma}$ and $\bar{\Iwf}=\bar{\Iwf}^{\Sigma,\gamma}$. We will prove some combinatorial properties of $\langle L,\bar{\Iwf}\rangle$ which are necessary for our isomorphism-of-names arguments on a $(\mathcal{S},\theta)$-standard iteration along $\langle L,\bar{\Iwf}\rangle$.

\begin{lemma}\label{smallTemp}
   If $A\subseteq L$ has size less than $\theta$, then $|\Iwf(A)|<\theta$.
\end{lemma}
\begin{proof}
   Without loss of generality, we can assume that $A$ is a tree. It is easy to note that $\{A\cap L_\alpha : \alpha\leq\gamma\}$ has size $<\theta$. To see that $\{A\cap J^{\Sigma,\gamma}_x : x\textrm{\ is }\Sigma\textrm{-relevant}\}$ has size less than $\theta$, note that if $x$ is $\Sigma$-relevant and $A\cap J^{\Sigma,\gamma}_x\neq\emptyset$, then $x':=x\frestr(|x|-1|)\in A$ and $\{A\cap J^{\Sigma,\gamma}_{x'^\smallfrown\langle\xi\rangle} : \xi\in\theta\}=
   \{A\cap J^{\Sigma,\gamma}_{x'^\smallfrown\langle\xi\rangle} : \xi<\rho\}$ for some $\rho<\theta$. Therefore $\Iwf(L)\frestr A$ has size $<\theta$ and so $\Iwf(A)$.
\end{proof}

For a $(\mathcal{S},\theta)$-standard iteration $\Por\frestr\langle L,\bar{\Iwf}\rangle$ where $L_C=\emptyset$ (as in Shelah's original construction), the produced poset only depends on the template structure. That is, if $A,B\subseteq L$ are isomorphic as linear orders, as trees and as templates (more precisely if they satisfy conditions (i)-(ix) of Definition \ref{DefItIsom} below), then $\Por\frestr A$ and $\Por\frestr B$ are isomorphic partial orders. An isomorphism between them can be constructed canonically from an isomorphism between $A$ and $B$. However, if $L_C\neq\emptyset$, such an isomorphism does not necessarily exist.

\begin{definition}\label{DefItIsom}
   Let $\Por\frestr\langle L,\bar{\Iwf}\rangle$ be a $(\mathcal{S},\theta)$-standard iteration as in Definition \ref{DefSigmaIt}. Say that $A\subseteq L$ is \emph{c.i.s. (closed-in-support with respect to $\Por\frestr\langle L,\bar{\Iwf}\rangle$)} if for any $x\in A\cap L_C$ we have $C_x\subseteq A$. We abbreviate \emph{closed-in-support tree} as \emph{c.i.s.t.}.

   If $A,B\subseteq L$ are c.i.s.t., they are \emph{$\Por\frestr\langle L,\bar{\Iwf}\rangle$-isomorphic} if there exists a \emph{$\Por\frestr\langle L,\bar{\Iwf}\rangle$-isomorphism} $h:A\to B$, that is, a bijection that satisfies, for all $x,y\in A$:
   \begin{enumerate}[(i)]
      \item $|h(x)|=|x|$,
      \item $h(x)\frestr k=h(x\frestr k)$ for all $0<k\leq|x|$,
      \item $x<y$ iff $h(x)<h(y)$,
      \item for $k<|x|$, $x(k)$ is positive iff $h(x)(k)$ is positive,
      \item if $|x|=|y|=k+1$, $x\frestr k=y\frestr k$ and $y(k)=x(k)+1$ is positive, then $h(y)(k)=h(x)(k)+1$,
      \item the dual of the previous statement with $y(k)$ negative, that is, if  $x\frestr k=y\frestr k$ and $y(k)=x(k)-1$ is negative, then $h(y)(k)=h(x)(k)-1$,
      \item if $\{x_\xi\}_{\xi<\beta}$ is a sequence in $A$, $z\in A$, $|z|=k+1$, $|x_\xi|=k+1$ and $x_\xi\frestr k=z\frestr k$ for any $\xi<\beta$ and $\{x_\xi(k)\}_{\xi<\beta}$ is an increasing sequence of positive ordinals with limit $z(k)$, then $h(z)(k)$ is the limit of $\{h(x_\xi)(k)\}_{\xi<\beta}$,
      \item the dual of the previous statement for a decreasing sequence of negative ordinals,
      \item $\Iwf_{h(x)}\frestr B=\{h[X] : X\in\Iwf_x\frestr A\}$ for all $x\in A$,
      \item for all $\eta<\nu$, $h[A\cap L_{S,\eta}]=B\cap L_{S,\eta}$,
      \item if $x\in L_C\cap A$ then $h[C_x]=C_{h(x)}$ and, whenever there is a sequence $\langle\hat{h}_D : D\subseteq C_x\rangle$ of functions such that
      \begin{enumerate}[(1)]
      \setcounter{enumii}{-1}
       \item $\hat{h}_D:\Por\frestr D\to\Por\frestr h[D]$ is an isomorphism,
       \item $X\subseteq D$ implies $\hat{h}_X\subseteq\hat{h}_D$,
       \item for $z\in D\cap L_C$ and $E\in\Pwf(D)\cap\hat{\Iwf}_z$, $\Qnm^{h[E]}_{h(z)}$ is the name associated to $\Qnm_z^E$ via $\hat{h}_E$ and,
       \item for $p\in\Por\frestr D$, $\dom(\hat{h}_D(p))=h[\dom p]$ and, if $z=\max(\dom p)$, $E\in\Iwf_z\frestr D$, $p\frestr L_z\in\Por\frestr E$ and $p(z)$ is a $\Por\frestr E$-name for a member of $\Qnm_z^E$, then $\hat{h}_D(p)\frestr L_{h(z)}=\hat{h}_E(p\frestr L_z)$ and $\hat{h}_D(p)(h(z))$ is the name associated to $p(z)$ via $\hat{h}_E$,
      \end{enumerate}
      then $\Qnm_{h(x)}$ is the name associated to $\Qnm_{x}$ via $\hat{h}_{C_x}$.
   \end{enumerate}
   By Corollary \ref{IsomTemp} there exists an isomorphism $\hat{h}:\Por\frestr A\to\Por\frestr B$ such that $\langle\hat{h}\frestr(\Por\frestr D):D\subseteq A\rangle$ is the unique sequence satisfying (0)-(3) above.
\end{definition}

We need to guarantee that for subsets of $L$ of size $<\theta$ there are only a few isomorphism-types.

\begin{lemma}\label{notmanytypes}
   If $\theta^{<\theta}=\theta$ and $\Por\frestr\langle L,\bar{\Iwf}\rangle$ is a $(\mathcal{S},\theta)$-standard iteration as in Definition \ref{DefSigmaIt}, then there are at most $\theta$-many different types of $\Por\frestr\la L,\bar{\Iwf}\ra$-isomorphic c.i.s. subtrees of $L$ of size $<\theta$.
\end{lemma}
\begin{proof}
   Given a c.i.s.t. $A\subseteq L$, we can find a tree $T\subseteq L^{\theta,\theta}$ of size $<\theta$ and a function $h:A\to T$ satisfying (i)-(viii) of Definition \ref{DefItIsom}. Let $\bar{\Jwf}$ be the template on $T$ such that $\Jwf_{h(x)}=\{h[X]:X\in\Iwf_x\frestr A\}$ for all $x\in A$.
   The function $h$ allows us to partition $T$ into the sets $T_C=h[A\cap L_C]$ and $T_{S,\eta}=h[A\cap L_{S,\eta}]$ for $\eta<\nu$ and to construct a $(\Swf,\theta)$-standard iteration along $\langle T,\bar{\Jwf}\rangle$ isomorphic (in the sense of Corollary \ref{IsomTemp}) to $\Por\frestr\la A,\bar{\Iwf}\frestr A\ra$ by Lemma \ref{InnConstr}. Here, note that $|\Jwf(T)|<\theta$ by Lemma \ref{smallTemp}.

   Therefore, it is enough to prove that there are $\theta$-many $(\mathcal{S},\theta)$-standard iterations along subtrees of $L^{\theta,\theta}$ of size $<\theta$ with a template structure that has $<\theta$ sets. As $\theta^{<\theta}=\theta$, there are $\theta$-many subtrees of $L^{\theta,\theta}$ size $<\theta$, so fix $T$ one of them. Now, there are at most $((2^{|T|})^{<\theta})^{|T|}$-many indexed templates $\bar{\Jwf}$ of $(T,<)$ such that $|\Jwf(T)|<\theta$. On the other hand, we can partition $T$ into pieces of the form $\{T_{S,\eta}\}_{\eta<\nu}\cup\{T_C\}$ in $(\nu+1)^{|T|}$-many ways (recall that $\nu\leq\theta$). After fixing one such indexed template and one such partition, there are at most $(2^{|T|})^{|T_C|}$-ways to choose a sequence $\la C'_x\ra_{x\in T_C}$ where each $C'_x\in\hat{\Jwf}_x$ and we fix one such sequence.

   According to Definition \ref{DefSigmaIt}, for fixed $T_{S,\eta}$ ($\eta<\nu$), $T_C$ and $\la C'_x\ra_{x\in T_C}$, a $(\Swf,\theta)$-standard iteration $\Por\frestr\la T,\bar{\Jwf}\ra$ depends only on the choice of the ordinals $\gamma_x<\theta$ and the $\Por\frestr C'_x$-names for $\sigma$-linked partial orders for $\gamma_x$. There are $\theta^{|T_C|}=\theta$-many choices of $\la\gamma_x\ra_{x\in T_C}$ so, fixing one of these choices, we show by induction on $\mathrm{Dp}^{\bar{\Jwf}}(Y)$ for $Y\subseteq T$ that there are at most $\theta$-many $(\Swf,\theta)$-standard iterations along $\la Y,\bar{\Jwf}\frestr Y\ra$ and that the poset produced by such an iteration has size $\leq\theta$. Consider cases on $Y$ according to Theorem \ref{TempIt}(d).
   \begin{itemize}
      \item \emph{$Y$ has a maximum $z$ and $Y_z=Y\cap T_z\in\hat{\Jwf}_z$.} Then, any desired standard iteration has the form $\Por\frestr Y=\Por\frestr Y_z\ast\Qnm^{Y_z}_z$. If $z\in T_S$ then the choice of $\Qnm^{Y_z}_z$ is fixed and there are as many $(\mathcal{S},\theta)$-standard iterations along $Y$ as there are along $Y_z$, which by the induction hypothesis are $\leq\theta$ and, as $\Por\frestr Y_z$ has size $\leq\theta$, it forces the continuum $\leq\theta$, so $\Por\frestr Y$ has size $\leq\theta$; if $z\in T_C$ and $C'_z\subseteq Y_z$, as $|\Por\frestr C'_z|\leq\theta$ and $\theta^{<\theta}=\theta$, then there are at most $\theta$-many (nice) $\Por\frestr C'_z$-names for partial orders for $\gamma_z$. Therefore, there are at most $\theta$-many $(\mathcal{S},\theta)$-standard iterations along $Y$. The case $C'_z\nsubseteq Y_z$ is easy.
      \item \emph{$Y$ has a maximum $z$ but $Y_z\notin\hat{\Jwf}_z$.} Here, a $(\mathcal{S},\theta)$-standard iteration along $Y$ satisfies $\Por\frestr Y=\limdir_{X\in\mathcal{B}}\Por\frestr X$ where $\mathcal{B}:=\{X\subseteq Y: X\cap T_z\in\Jwf_z\frestr Y\}$. $\mathcal{B}$ has size $<\theta$ because $|\Jwf_z\frestr Y|\leq|\Jwf(T)\frestr Y|<\theta$ so, by the induction hypothesis, there are at most $\theta^{<\theta}=\theta$-many ways to define $\Por\frestr Y$.
      \item \emph{$Y$ does not have a maximum.} A similar argument as in the previous case works.
   \end{itemize}
\end{proof}


\section{Preservation properties}\label{SecPresProp}


The preservation properties discussed in this section were developed for fsi of ccc posets by Judah and Shelah \cite{jushe}, with improvements by Brendle \cite{Br-Cichon}. These are summarized and generalized in \cite{gold} and in \cite[Sect. 6.4 and 6.5]{judabarto}. The presentation in this section is based on \cite{mejia,mejia-temp}.

\begin{context}\label{ContextUnbd}
 Fix an increasing sequence $\langle\sqsubset_n\rangle_{n<\omega}$ of 2-place closed relations (in the topological sense) in $\omega^\omega$ such that for any $n<\omega$ and $g\in\omega^\omega$, $(\sqsubset_n)^g=\left\{f\in\omega^\omega : f\sqsubset_n g\right\}$ is (closed) nwd (nowhere dense).

 Put $\sqsubset=\bigcup_{n<\omega}\sqsubset_n$. Therefore, for every $g\in\omega^\omega$, $(\sqsubset)^g$ is an $F_\sigma$ meager set.

 For $f,g\in\omega^\omega$, say that \emph{$g$ $\sqsubset$-dominates $f$} if $f\sqsubset g$. $F\subseteq\omega^\omega$ is a \emph{$\sqsubset$-unbounded family} if no function in $\omega^\omega$ $\sqsubset$-dominates all the members of $F$. Associate with this notion the cardinal $\bfrak_\sqsubset$, which is the least size of a $\sqsubset$-unbounded family. Dually, say that $C\subseteq\omega^\omega$ is a \emph{$\sqsubset$-dominating family} if any real in $\omega^\omega$ is $\sqsubset$-dominated by some member of $C$. The cardinal $\dfrak_\sqsubset$ is the least size of a $\sqsubset$-dominating family. Given a set $Y$, say that a real $f\in\omega^\omega$ is \emph{$\sqsubset$-unbounded over $Y$} if $f\not\sqsubset g$ for every $g\in Y\cap\omega^\omega$.
\end{context}

Context \ref{ContextUnbd} is defined for $\omega^\omega$ for simplicity, but in general the same notions apply by changing the space for the domain or the codomain of $\sqsubset$ to another uncountable Polish space whose members can be coded by reals in $\omega^\omega$.

From now on, fix $\theta_0$ an uncountable regular cardinal.

\begin{definition}[Judah and Shelah {\cite{jushe}}, {\cite[Def. 6.4.4]{judabarto}}]\label{DefPresProp}
   A forcing notion $\Por$ is \emph{$\theta_0$-$\sqsubset$-good} if the following property holds\footnote{\cite[Def. 6.4.4]{judabarto} has a different formulation, which is equivalent to our formulation for $\theta_0$-cc posets (recall that $\theta_0$ is uncountable regular). See \cite[Lemma 2]{mejia} for details.}: For any $\Por$-name $\dot{h}$ for a real in $\omega^\omega$ there exists a nonempty $Y\subseteq\omega^\omega$ (in the ground model) of size $<\theta_0$ such that for any $f\in\omega^\omega$ which is $\sqsubset$-unbounded over $Y$, we have $\Vdash f\not\sqsubset\dot{h}$. A forcing notion is said to be  \emph{$\sqsubset$-good}, if it is $\aleph_1$-$\sqsubset$-good.
\end{definition}

This is a standard property intended to preserve $\bfrak_\sqsubset$ small and $\dfrak_\sqsubset$ large in forcing extensions. A subset $F$ of  $\omega^\omega$ is said to be \emph{$\theta_0$-$\sqsubset$-unbounded} if for any $X\subseteq\omega^\omega$ of size $<\theta_0$, there exists an $f\in F$ which is $\sqsubset$-unbounded over $X$. Clearly, if $F$ is such a family, then $\bfrak_\sqsubset\leq|F|$ and $\theta_0\leq\dfrak_\sqsubset$. On the other hand,
$\theta_0$-$\sqsubset$-unbounded families of the ground model remain such in generic extensions of $\theta_0$-$\sqsubset$-good posets. Thus, if $\lambda\geq\theta_0$ is a cardinal and $\dfrak_\sqsubset\geq\lambda$ in the ground model, then the inequality is preserved by such generic extension. It is also known that the property of Definition \ref{DefPresProp} is preserved under fsi of $\theta_0$-cc posets. Also, if $\Por\lessdot\Qor$ and $\Qor$ is $\theta_0$-$\sqsubset$-good, then so is $\Por$.

\begin{lemma}[{\cite[Lemma 4]{mejia}}]\label{smallPlus}
   Every poset of size $<\theta_0$ is $\theta_0$-$\sqsubset$-good. In particular, $\Cor$ is $\sqsubset$-good.
\end{lemma}

\begin{example}\label{SubsecUnbd}
 \begin{enumerate}[(1)]


  \item \emph{Preserving splitting families:} For $A,B\in[\omega]^\omega$ and $n<\omega$, define $A\propto_n B$ iff either $B\menos n\subseteq A$ or $B\menos n\subseteq\omega\menos A$, so $A\propto B\sii(B\subseteq^* A\textrm{\ or }B\subseteq^*\omega\menos A)$. Note also that $A\not\propto B$ iff $A$ splits $B$, so $\sfrak=\bfrak_\propto$ and $\rfrak=\dfrak_\propto$. Baumgartner and Dordal \cite{baudor} proved that $\Dor$ is $\propto$-good (see also \cite[Main Lemma 3.8]{brendlebog}).



  \item \emph{Preserving null-covering families:} Let $\langle I_k\rangle_{k<\omega}$ be the interval partition of $\omega$ such that $|I_k|=2^{k+1}$ for all $k<\omega$. For $n<\omega$ and $f,g\in2^\omega$ define $f\pitchfork_ng\sii(\forall k\geq n)(f\frestr I_k\neq g\frestr I_k)$ and let $f\pitchfork g\sii$ for all but finitely many $k$ we have $f\frestr I_k\neq g\frestr I_k$. Clearly, $(\pitchfork)^g$ is a co-null $F_\sigma$ meager set. This relation is related to the cardinal characteristics of covering and uniformity of the null ideal, because $\cov(\Nwf)\leq\bfrak_\pitchfork\leq\non(\Mwf)$ and $\cov(\Mwf)\leq\dfrak_\pitchfork\leq\non(\Nwf)$ (see \cite[Lemma 7]{mejia}). By \cite[Lemma $1^*$]{Br-Cichon} for every infinite cardinal $\nu<\theta_0$, $\nu$-centered forcing notions are  $\theta_0$-$\pitchfork$-good.

  \item \emph{Preserving ``union of null sets is non-null'':} Fix $\Hwf:=\{id^{k+1} : k<\omega\}$ (where $id^{k+1}(i)=i^{k+1}$) and let $S(\omega,\Hwf):=\bigcup_{h\in\Hwf}S(\omega,h)$. For $n<\omega$, $x\in\omega^\omega$ and a slalom $\psi\in S(\omega,\Hwf)$, let $x\in^*_n\psi$ iff $(\forall k\geq n)(x(k)\in\psi(k))$, so $x\in^*\psi$ iff  for all but finitely many $k$ we have $x(k)\in\psi(k)$. By Bartoszy\'{n}ski's characterization (see Subsection \ref{SubSecCard}) applied to $id$ and to a function $g$ that dominates all the functions in $\Hwf$ we obtain $\add(\Nwf)=\bfrak_{\in^*}$ and $\cof(\Nwf)=\dfrak_{\in^*}$. Judah and Shelah \cite{jushe} proved that given an infinite cardinal $\nu<\theta_0$, every $\nu$-centered forcing notion is $\theta_0$-$\in^*$-good. Moreover, as a consequence of results of Kamburelis \cite{kamburelis}, any subalgebra\footnote{Here, $\Bor$ is seen as the complete Boolean algebra of Borel sets (in $2^\omega$) modulo the null ideal.} of $\Bor$ is $\in^*$-good.


 \end{enumerate}
\end{example}

We recall the following preservation result for template iterations.


\begin{theorem}\label{PresTemp2}
   Let $\Por\frestr\langle L,\bar{\Iwf}\rangle$ be a template iteration such that $L$ does not have a maximum, $[L]^{<\omega}\subseteq\Iwf(L)$ and $\Por\frestr L$ is $\theta_0$-cc. Assume, for any $A\in\Iwf(L)\menos\{\emptyset\}$:
   \begin{enumerate}[(i)]
      \item if $A$ has a maximum $x$ and $A_x:=A\cap L_x\in\hat{\Iwf}_x$, then $A_x\in\Iwf_x$;
      \item if $A$ has a maximum $x$, $A_x:=A\cap L_x\notin\hat{\Iwf}_x$ and $\dot{h}$ is a $\Por\frestr A$-name for a real, then there exists an increasing sequence $\langle B_n\rangle_{n<\omega}$ in $\Bwf_A:=\{B\subseteq A: B\cap L_x\in\Iwf_x\frestr A\}$ such that $\dot{h}$ is a $\Por\frestr C$-name for a real,
            where $C:=\bigcup_{n<\omega}B_n$, and $\Por\frestr C=\limdir_{n<\omega}\Por\frestr B_n$;
      \item if $A$ does not have a maximum and $\dot{h}$ is a $\Por\frestr A$-name for a real, then
            there exists an increasing sequence $\langle B_n\rangle_{n<\omega}$ in $\Bwf_A:=\{B\in\Iwf_x\frestr A: x\in A\}$ like in (ii);
      \item for all $x\in L$ and $B\in\Iwf_x$, $\Vdash_{\Por\upharpoonright B}\Qnm^B_x$ is $\theta_0$-$\sqsubset$-good.
   \end{enumerate}
   Then, $\Por\frestr L$ is $\theta_0$-$\sqsubset$-good.
\end{theorem}
\begin{proof}
   The proof is the same as \cite[Thm. 5.10]{mejia-temp}, but in this case, prove by induction on $\rank_{\Iwf(L)}(A)$ for $A\in\Iwf(L)$ that $\Por\frestr A$ is $\theta_0$-$\sqsubset$-good.
\end{proof}

Until the end of this section, fix $\gamma,\delta,\tau$ non-zero ordinals, $\delta$ and $\tau$ with uncountable cofinality, $L=L^{\delta,\gamma}$, $\Sigma=\langle S_\beta : \beta<\tau\rangle$ a partition of $\delta^*$, $\Iwf=\Iwf^{\Sigma,\gamma}$, $\bar{\Iwf}=\bar{\Iwf}^{\Sigma,\gamma}$ and $\Iwf_x=\Iwf_x^{\Sigma,\gamma}$. For $x\in L$ $\Sigma$-relevant, denote $J_x=J^{\Sigma,\gamma}_x$. Recall that any member of $\Iwf$ is written as a unique finite disjoint union of basic sets (see Definition \ref{DefShelahTemp}). For $a\in L$, denote by $[a]^-$ the set of sequences $x\in L$ such that $x$ end-extends $a\rest(|a|-1)$, $|x|\geq|a|$ and $x(|a|-1)<a(|a|-1)$. Denote by $[a]^+$ the set of sequences in $L$ that end-extend $a\rest(|a|-1)$ but are not in $[a]^-$ (that is, $x\in[a]^+$ iff either $x=a\rest(|a|-1)$, or $|x|\geq|a|$, $x$ end-extends $a\rest(|a|-1)$ and $x(|a|-1)\geq a(|a|-1)$).

\begin{theorem}\label{ThmItPres}
   Let $\Por\langle L,\bar{\Iwf}\rangle$ be a template iteration, and suppose $\Por\frestr L$ has the ccc. Assume that for all $x\in L$ and $B\in\Iwf_x$, $\Vdash_{\Por\upharpoonright B}\Qnm^B_x$ is $\theta_0$-$\sqsubset$-good.
   Then, $\Por\frestr L$ is $\theta_0$-$\sqsubset$-good.
\end{theorem}
\begin{proof}
   By Lemma \ref{TempRealSupp} (see below) the conditions of Theorem \ref{PresTemp2} are satisfied (note that condition (ii) there is irrelevant).
\end{proof}

\begin{lemma}\label{disjextLemma}
   Let $a\in L$ with $|a|\geq 2$ and $a(|a|-1)=0$, $\Bwf$ a countable collection of basic sets contained in $[a]^-$ such that no initial segment of $L$ is in $\Bwf$.\footnote{This assumption is relevant only when $a=\la0,0\ra$ because $[a]^-=L_0$. Otherwise, $[a]^-$ does not contain basic sets which are initial segments.} Then, there is a countable collection $\Ewf$ of pairwise disjoint basic sets contained in $[a]^-$ such that
   \begin{enumerate}[(a)]
      \item any member of $\Bwf$ is contained in a (unique) member of $\Ewf$,
      \item any member of $\Ewf$ contains some member of $\Bwf$,
      \item $\Ewf$ does not contain initial segments of $L$ and
      \item no pair of members of $\Ewf$ are contained in any basic set included in $[a]^-$ that is not an initial segment of $L$.
   \end{enumerate}
   Furthermore, the same statement holds when $[a]^-$ is replaced by $[a]^+$.
\end{lemma}
\begin{proof}
   For $B\in\Bwf$ let $x_B$ be the unique member of $[a]^-$ such that either $B=J_{x_B}$ or $B=\{x_B\}$ where, in the first case, $x_B$ is $\Sigma$-relevant. Define $z_B$ according to those two cases: in the first case, $z_B=x_B\frestr m$ where $m\geq|a|$ is minimal such that $x_B\frestr(m+1)$ is $\Sigma$-relevant; in the second case, let $z_B=x_B\frestr m$ where $m\geq|a|$ is minimal (if exists) such that either $m<|x_B|$ and ${x_B}^\smallfrown\langle\max\{0,x_B(m)\}\rangle$ is $\Sigma$-relevant, or $m=|x_B|$ and ${x_B}^\smallfrown\{0\}$ is $\Sigma$-relevant, otherwise, if there is no such $m$, put $z_B=\emptyset$.

   Let $H=\{z_B : B\in\Bwf\}\menos\{\emptyset\}$ which is a subset of $[a]^-$. For each $y\in H$, let $y'=y^\smallfrown\{\eta_y\}$ where $$\eta_y=\sup(\{0\}\cup\{x_B(|z_B|)+1: B\in H,\ z_B=y,\ |z_B|<|x_B|\textrm{\ and }x_B(|z_B|)\geq0\}).$$ As $\Bwf$ is countable and $\delta,\tau$ have uncountable cofinalities, then $\eta_y<\min\{\delta,\tau\}$ so $y'\in L$ (even in $[a]^-$ with length larger than $|a|$) and it is $\Sigma$-relevant. $\Ewf=\{J_{y'} : y\in H\}\cup\{B\in\Bwf : z_B=\emptyset\}$ is as desired.

   The same argument works for $[a]^+$.
\end{proof}

\newcommand{\cf}{\mathrm{cf}}

\begin{lemma}\label{TempRealSupp}
   For $A\in\Iwf\menos\{\emptyset\}$:
   \begin{enumerate}[(a)]
      \item If $x=\max(A)$ then $A\cap L_x\in\Iwf_x$.
      \item Let $\Por\langle L,\bar{\Iwf}\rangle$ be a template iteration, and suppose $\Por\frestr L$ has the ccc. If $A$ does \underline{not} have a maximum and $\dot{h}$ is a $\Por\frestr A$ name for a real, then there exists an increasing sequence $\langle B_n\rangle_{n<\omega}$ in $\Awf:=\{B\in\Iwf_x\frestr A : x\in A\}$ such that
            \begin{enumerate}[(i)]
              \item $\dot{h}$ is a $\Por\frestr C$-name, where $C:=\bigcup_{n<\omega}B_n$, and
              \item $\Por\frestr C$ is the direct limit of $\la\Por\frestr B_n\ra_{n<\omega}$.
            \end{enumerate}
   \end{enumerate}
\end{lemma}
\begin{proof}
   Note that the only basic sets of $\Iwf$ that have a maximum are the singletons. Therefore, if $A\in\Iwf$ and $x=\max(A)$, it is clear that $A\menos\{x\}$ is still a union of basic sets of
   $\Iwf$, so (a) holds.

   We prove (b). If $\dot{h}$ is a $\Por\frestr B$-name for some $B\in\Awf$, then $B_n:=B$ works, so we assume that this is not the case. As $A\in\Iwf\menos\{\emptyset\}$, $A=\bigcup_{k\leq M}E_k$ for some $M<\omega$ and $\{E_k\}_{k\leq M}$ a sequence of basic sets of $\Iwf$ such that $E_k<E_{k+1}$ (that is, every member of $E_k$ is less than every member of $E_{k+1}$) for $k<M$. $E_M$ cannot be a singleton because $A$ does not have a maximum.

   Given $\dot{h}$ a $\Por\frestr A$-name for a real in $\omega^\omega$, by ccc-ness there is a set of conditions $\{p_n : n<\omega\}$ in $\Por\frestr A$ determining the name $\dot{h}$ (i.e. the union of the maximal antichains that decide $\dot{h}(i)$ for each $i<\omega$). Then, for each $n<\omega$, there exists a $C_n\in\Awf$ such that $p_n\in\Por\frestr C_n$, without loss of generality, $\bigcup_{k<M}E_k\subseteq C_n$. By cases on $E_M$ we construct a sequence $\la B_n\ra_{n<\omega}$ of sets in $\Awf$ such that
   \begin{itemize}
     \item[(*)] for any $x\in A$ and $H\in\Iwf_x\rest A$ there is an $n_0<\omega$ such that $H\cap\bigcup_{n<\omega}B_n=H\cap B_{n_0}$.
   \end{itemize}
   Note that it is enough to prove (*) when $H\subseteq A\cap L_x$ is a basic set.
   \begin{enumerate}[(1)]
      \item $E_M=L_\xi$ for some $\xi\leq\gamma$, which implies $M=0$. Consider the following cases.
         \begin{itemize}
            \item $\xi=0$. For $n<\omega$, let $\Cwf_n$ be the family of pairwise disjoint basic sets of the (unique) decomposition of $C_n$, which are clearly contained in $[\langle 0,0\rangle]^-$. Put $\Cwf=\bigcup_{n<\omega}\Cwf_n$ and find $\Ewf$ as in Lemma \ref{disjextLemma} applied to $\Cwf$. $\Ewf$ is infinite (if not, $\dot{h}$ is a $\Por\frestr B$-name for some $B\in\Awf$), so enumerate $\Ewf=\{H_k : k<\omega\}$ and put $B_n=\bigcup_{k\leq n}H_k$ for $n<\omega$.

                (*) holds because, if $x\in L_0$ and $H\subseteq L_x$ is basic, then $H\subseteq[\la0,0\ra]^-$ is not an initial segment and $H$ intersects at most one $H_k$ by Lemma \ref{disjextLemma} (recall that, if two basic sets have non-empty intersection, then one of them is contained in the other).
            \item $\xi=\eta+1$. We may assume that $L_\eta\subseteq C_n$ for all $n<\omega$. Then, the disjoint decomposition of $C_n$ into basic sets are $L_\eta$ and subsets of either $[\langle\eta,0\rangle]^+$ or $[\langle\eta+1,0\rangle]^-$. Let $\Cwf_n^0$ be the family of these basic sets contained in $[\langle\eta,0\rangle]^+$ and, similarly, let $\Cwf_n^1$ be the family corresponding to $[\langle\eta+1,0\rangle]^-$. Put $\Cwf^i=\bigcup_{n<\omega}\Cwf^i_n$ and let $\Ewf^i$ be as in Lemma \ref{disjextLemma} applied to $\Cwf^i$ for $i\in\{0,1\}$. Put $\Ewf=\Ewf^0\cup\Ewf^1$, which is infinite. Enumerate $\Ewf=\{H_k : k<\omega\}$ and put $B_n=L_\eta\cup\bigcup_{k\leq n}H_k$ for $n<\omega$.

                Now let $x\in L_{\eta+1}$ and $H\subseteq L_x$ be basic. If $H$ intersect $L_\eta$ then it must be contained in it so $n_0=0$ works for (*); if $H\cap L_\eta\cap\emptyset$ then either $H\subseteq[\langle\eta,0\rangle]^+$ or $H\subseteq[\langle\eta+1,0\rangle]^-$, but in any case $H$ intersects at most one $H_k$. Thus, (*) holds.
            \item $\xi$ is a limit ordinal. We may assume that, for $n<\omega$, the disjoint decomposition of $C_n$ into basic sets are $L_{\alpha_n}\in\Awf$, for some $\alpha_n<\xi$, and basic subsets of $[\langle\xi,0\rangle]^-$. Let $\Cwf_n$ be the family of the latter basic sets. Without loss of generality, if $\mathrm{cf}(\xi)=\omega$ then $\{\alpha_n\}_{n<\omega}$ is an increasing sequence of ordinals converging to $\xi$, otherwise, the sequence is constant $\alpha$. Put $\Cwf=\bigcup_{n<\omega}\Cwf_n$ and find $\Ewf$ by Lemma \ref{disjextLemma} applied to $\Cwf$. $\Ewf=\{H_k : k<\nu\}$ for some $\nu\leq\omega$ ($\nu=\omega$ when $\cf(\xi)>\omega$), so put $B_n=L_{\alpha_n}\cup\bigcup_{k<\min\{n+1,\nu\}}H_k$ for $n<\omega$.

                Let $x\in L_\xi$ and $H\subseteq L_x$ be basic. If $H$ intersects $[\xi]^-=\{x\in L:x(0)<\xi\}$ then $H$ is contained in it. If $\cf(\xi)=\omega$ then $H$ is contained in some $L_{\alpha_i}$ so $n_0$ can be found as in (*), else, $n_0=0$ works when $\cf(\xi)>\omega$; if $H\cap[\xi]^-=\emptyset$ then $H\subseteq[\langle\xi,0\rangle]^-$ so $H$ intersects at most one $H_k$ and $n_0$ as in (*) can be found.
         \end{itemize}
      \item $E_M=J_x$ for some $\Sigma$-relevant $x$. Let $m=|x|$. In each of the following cases (*) can be proven as before. We just show (*) for the last case.
          \begin{itemize}
            \item $x(m-1)=0$. For $n<\omega$,  let $\{E_k: k<M\}\cup\Cwf^0_n\cup\Cwf^1_n$ be the decomposition of $C_n$ into disjoint basic sets, where $\Cwf^0_n\subseteq[x]^-$ and $\Cwf^1_n\subseteq[x^\smallfrown\la0\ra]^-$. Put $\Cwf^i=\bigcup_{n<\omega}\Cwf^i_n$ and find $\Ewf^i$ as in Lemma \ref{disjextLemma} applied to $\Cwf^i$ for each $i\in\{0,1\}$. $\Ewf=\Ewf^0\cup\Ewf^1$ is infinite (if not, $\dot{h}$ is a $\Por\frestr B$-name for some $B\in\Awf$), so enumerate $\Ewf=\{H_k : k<\omega\}$ and put $B_n=\bigcup_{k<M}E_k\cup\bigcup_{k\leq n}H_k$ for $n<\omega$.
            \item $x(m-1)=\eta+1$. Let $x^0=x\frestr(m-1)^\smallfrown\{\eta\}$ and $x^1=x$. We may assume that $J_{x^0}\subseteq C_n$ for all $n<\omega$. Then, the disjoint decomposition of $C_n$ into basic sets are $E_k$, for $k<M$, $J_{x^0}$ and subsets of either $[{x^0}^\smallfrown\la0\ra]^+$ or $[{x^1}^\smallfrown\la0\ra]^-$. Let $\Cwf_n^0$ be the family of these basic sets contained in $[{x^0}^\smallfrown\la0\ra]^+$ and define $\Cwf^1$ likewise. Put $\Cwf^i=\bigcup_{n<\omega}\Cwf^i_n$ and let $\Ewf^i\subseteq[x^i]$ be as in Lemma \ref{disjextLemma} applied to $\Cwf^i$ for $i\in\{0,1\}$. Put $\Ewf=\Ewf^0\cup\Ewf^1$, which is infinite. Enumerate $\Ewf=\{H_k : k<\omega\}$ and put $B_n=\bigcup_{k<m}E_k\cup J_{x^0}\cup\bigcup_{k\leq n}H_k$ for $n<\omega$.
            \item $x(m-1)$ is a limit ordinal. We may assume that, for $n<\omega$, the disjoint decomposition of $C_n$ into basic sets are $E_k$, for $k<M$, $J_{x^n}$ where $x^n=x\upharpoonright(m-1)^\smallfrown\{\alpha_n\}$ for some $\alpha_n<x(m-1)$, and basic subsets of $[x^\smallfrown\la0\ra]^-$. Let $\Cwf_n$ be the family of the latter basic sets. Without loss of generality, if $\cf(x(m-1))=\omega$ then $\{\alpha_n\}_{n<\omega}$ is an increasing sequence with limit $x(m-1)$, otherwise, the sequence is constant $\alpha$ (so $\la x^n\ra_{n<\omega}$ is also constant). Put $\Cwf=\bigcup_{n<\omega}\Cwf_n$ and find $\Ewf$ by Lemma \ref{disjextLemma} applied to $\Cwf$. $\Ewf=\{H_k : k<\nu\}$ for some $\nu\leq\omega$ ($\nu=\omega$ when $\cf(x(m-1))>\omega$), so put $B_n=\bigcup_{k<M}E_k\cup J_{x^n}\cup\bigcup_{k<\min\{n+1,\nu\}}H_k$ for $n<\omega$.

                To see (*), let $y\in A$ and $H\subseteq A\cap L_y$ be basic. If $H$ intersects $[x^\smallfrown\la0\ra]^-$ then $H$ is contained in it and intersects at most one $H_k$, so $n_0$ as in (*) exists; if $H\cap[x^\smallfrown\la0\ra]^-=\emptyset$ then it is clear that $n_0=0$ works when $\cf(x(m-1))>\omega$, otherwise, $H$ is contained in $\bigcup_{k<M}E_k\cup J_{x^{n_0}}$ for some $n_0<\omega$.
          \end{itemize}
   \end{enumerate}
   It is clear that $\{B_n : n<\omega\}\subseteq\Awf$ is $\subseteq$-increasing and that $\dot{h}$ is a $\Por\frestr C$-name (by Lemma \ref{disjextLemma}(a)), where $C=\bigcup_{n<\omega}B_n\supseteq\bigcup_{n<\omega}C_n$, so it remains to prove that $\Por\frestr C=\limdir_{n<\omega}\Por\frestr B_n$. Let $p\in\Por\frestr C$ and $x=\max(\dom p)$, so there exists a $D\in\Iwf_x\frestr C$ such that $p\frestr L_x\in\Por\frestr D$ and
   $p(x)$ is a $\Por\frestr D$-name of a member of $\Qnm_x^D$. Then, $D=C\cap H$ for some $H\in\Iwf_x$. By (*) applied to $A\cap H$, there exists an $n_0<\omega$ such that $B_{n_0}\cap H=B_{n_0}\cap(A\cap H)=C\cap(A\cap H)=D$ and $x\in B_{n_0}$, so $D\in\Iwf_x\frestr B_{n_0}$ which implies $p\in\Por\frestr B_{n_0}$.
\end{proof}

We will need the following results.

\begin{theorem}[{\cite[Thm. 5.17]{mejia-temp}}]\label{dfrakbig}
   Let $\theta$ be an uncountable regular cardinal and $\Por\frestr\langle L,\bar{\Iwf}\rangle$ a standard template iteration (see Definition \ref{DefStandardIt}). Assume:
   \begin{enumerate}[(i)]
    \item If $\dot{x}$ is a $\Por\frestr L$-name for a real, then it is a $\Por\frestr A$-name for some $A\subseteq L$ of size $<\theta$.
    \item For every $x\in L_S$ and $B\in\hat{\Iwf}_x$, $\Por\frestr B$ forces that $\Qnm^B_x$ is $\sqsubset$-good.
    \item $W\subseteq L$ is a cofinal subset of size $\lambda\geq\theta$ such that, for all $z\in W$, $L_z\in\Iwf_z$ and there is a $\Por\frestr(L_z\cup\{z\})$-name $\dot{c}_z$ for a $\sqsubset$-unbounded real over $V^{\Por\upharpoonright L_z}$.
   \end{enumerate}
   Then, $\Por\frestr L$ forces $\dfrak_\sqsubset\geq\lambda$.
\end{theorem}

\begin{theorem}[New reals not added at other stages {\cite[Thm. 5.12]{mejia-temp}}]\label{newrealint}
   Let $\Por\frestr\langle L,\bar{\Iwf}\rangle$ be a standard template iteration (see Definition \ref{DefStandardIt}),
   $x\in L$ such that $\bar{L}_x:=L_x\cup\{x\}\in\hat{\Iwf}_z$ for all $z>x$ in $L$ and let $\dot{f}$ be a $\Por\frestr\bar{L}_x$-name of a real such that
   $\Vdash_{\Por\upharpoonright\bar{L}_x}\dot{f}\notin V^{\Por\upharpoonright L_x}$. Then, $\Por\frestr L$ forces that $\dot{f}\notin V^{\Por\upharpoonright(L\menos\{x\})}$.
\end{theorem}


\section{Proof of the Main Theorem}\label{SecMain}

\begin{mainthm}
   Let $\theta_0\leq\theta_1\leq\theta<\mu<\lambda$ be uncountable regular cardinals with $\theta^{<\theta}=\theta$ and $\lambda^{<\lambda}=\lambda$. Then, there is a ccc poset that forces $\add(\Nwf)=\theta_0$, $\cov(\Nwf)=\theta_1$, $\pfrak=\sfrak=\gfrak=\theta$, $\add(\Mwf)=\cof(\Mwf)=\mu$ and $\afrak=\non(\Nwf)=\rfrak=\cfrak=\lambda$.
\end{mainthm}

Fix, throughout this section, $\theta_0\leq\theta_1\leq\theta<\mu<\lambda$ regular uncountable cardinals, such that  $\theta^{<\theta}=\theta$ and $\lambda^{<\lambda}=\lambda$. We may assume\footnote{This is forced by a fsi of length $\lambda$ where, by a book-keeping argument, all subposets of $\Loc$ of size $<\theta_0$, all subalgebras of $\Bor$ of size $<\theta_1$ and all posets of the form $\Mor_F$ for a filter base $F$ of size $<\theta$ are used along the iteration.} that there are
\begin{enumerate}[(I)]
  \item a $\theta_0$-$\in^*$-unbounded family of size $\theta_0$,
  \item a $\theta_1$-$\pitchfork$-unbounded family of size $\theta_1$ and
  \item a $\theta$-$\propto$-unbounded family of size $\theta$.
\end{enumerate}

Fix $\Sigma=\la S_\beta : \beta <\theta\ra$ a sequence of pairwise disjoint sets, each of which is co-initial in $\lambda^*$ and such that $\lambda^*=\bigcup_{\beta <\theta}S_\beta$. For $\delta\leq\lambda$, let $\la L^\delta,\bar{\mathcal{I}}^\delta\ra$ be the template defined as follows. Put $L^\delta= L^{\delta,\lambda\cdot\mu}$ as in Definition \ref{DefShelahTemp}, where $\lambda\cdot\mu$ denotes the product as ordinals and let $\Sigma_\delta=\la S_\beta\cap\delta^* : \beta<\theta\ra$. Define $\Iwf^\delta=\Iwf^{\Sigma_\delta,\lambda\cdot\mu}$ (see Definition \ref{DefShelahTemp}).

Note that $x\in L^\delta$ is $\Sigma_\delta$-relevant iff it is $\Sigma$-relevant. For shortness we just call such sequences \emph{relevant}. For such relevant $x$, we denote $J^\delta_x=J^{\delta,\Sigma_\delta}_x$. The sequence of templates $\la (L^\delta,\bar{\mathcal{I}}^\delta)\ra_{\delta\leq\lambda}$ has the following property.

\begin{lemma}\label{restTemp}
If $\theta\leq\delta \leq\delta^\prime\leq\lambda$ then $\mathcal{I}^\delta=\mathcal{I}^{\delta^\prime}\restrict L^\delta$. So for $x\in L^\delta$
we have $\mathcal{I}^\delta_x=\mathcal{I}_x^{\delta^\prime}\restrict L^\delta$.
\end{lemma}
\begin{proof} Observe that $L^\delta_\alpha=L^{\delta^\prime}_\alpha\cap L^\delta$ where $\alpha\in\lambda\mu$ and $J^{\delta^\prime}_x\cap L^\delta$ is either equal to $J^\delta_x$ when $x\in L^\delta$, or is the empty set when $x\notin L^\delta$.
\end{proof}

\begin{definition}\label{DefAppr}
An iteration $\Por\la L,\bar{\Iwf}\ra$ is called \emph{pre-appropriate} if it is a $(\la\Dor\ra,\theta)$-standard iteration where:
\begin{enumerate}[(1)]
  \item $\la L,\bar{\Iwf}\ra=\la L^\delta,\bar{\Iwf}^\delta\ra$ for some $0<\delta\leq\lambda$.
  \item $L=L_H\cup L_A\cup L_R\cup L_F$ is a disjoint union, $L_S=L_H$ and $L_C=L\menos L_H$.
  \item $L_H\cap\lambda\cdot\mu$ is cofinal in $\lambda\cdot\mu$ and has size $\lambda$.
  {\item If $x\in L_H$ then for $B\in \hat{\Iwf}_x$, $\dot{\mathbb{Q}}^B_x=\dot{\mathbb{D}}^{V^{\Por\restrict B}}$.}
  {\item For every $x\in L_F$ there are fixed $C_x\in\hat{\Iwf}_x$ of size $<\theta$ and a $\Por\restrict C_x$-name $\dot{F}_x$ for a filter base of size $<\theta$. $\Qnm_x=\Mor_{\dot{F}_x}$, that is, for $B\in \hat{\mathcal{I}}_x$,
  \[
    \dot{\mathbb{Q}}^B_x=
    \begin{cases}
       \mathbb{M}_{\dot{F}_x}, &\text{in case }C_x\subseteq B, \\
       \text{trivial poset}, &\text{in case }C_x\not\subseteq B.
    \end{cases}
  \]}
  {\item For $x\in L_R$, there are fixed $C_x\in\hat{\Iwf}_x$ of size $<\theta$ and $\Qnm_x$ is a $\Por\frestr C_x$-name of a subalgebra of $\Bor^{V^{\Por\upharpoonright C_x}}$ of size $<\theta_1$.}
  {\item For $x\in L_A$, there are fixed $C_x\in\hat{\Iwf}_x$ of size $<\theta$ and $\Qnm_x$ is a $\Por\frestr C_x$-name of a $\sigma$-linked subposet of $\Loc^{V^{\Por\upharpoonright C_x}}$ of size $<\theta_0$.}
\end{enumerate}
We call such an iteration \emph{appropriate} if it satisfies, additionally:
\begin{enumerate}[(1)]
\setcounter{enumi}{7}
\item If $\dot{F}$ is a $\Por\rest L$-name for a filter base of size $<\theta$, then there is $x\in L_F$ such that $\Vdash_{\Por\rest L} \dot{F}=\dot{F}_x$.
\item If $\Qnm$ is a $\Por\rest L$-name of a subalgebra of $\Bor$ of size $<\theta_1$, then there is an $x\in L_R$ such that $\Vdash_{\Por\rest L}\Qnm=\Qnm_x$.
\item If $\Qnm$ is a $\Por\rest L$-name of a $\sigma$-linked subposet of $\Loc$ of size $<\theta_0$, then there is an $x\in L_A$ such that $\Vdash_{\Por\rest L}\Qnm=\Qnm_x$.
\end{enumerate}
\end{definition}

\begin{lemma}\label{general_properties}
Let $\Por\rest\la L,\bar{\Iwf}\ra$ be a pre-appropriate iteration. If $A\subseteq L$, then
\begin{itemize}
\item $\mathbb{P}\restrict A$ has the Knaster property,
\item if $p\in\Por\restrict A$ then there is $C\subseteq A$ of cardinality (strictly) smaller than $\theta$ such that $p\in\Por\restrict C$ and
\item if $\dot{h}$ is a $\Por\restrict A$-name for a real then there is $C\subseteq A$ of cardinality (strictly) smaller than $\theta$ such that
$\dot{h}$ is a $\Por\restrict C$-name for a real.
\end{itemize}
\end{lemma}
\begin{proof}
   Directly from Lemma \ref{CondSupp}.
\end{proof}

\begin{lemma}\label{ApprLemma}
   If $\la L,\bar{\Iwf}\ra=\la L^\delta,\bar{\Iwf}^\delta\ra$ for some $\delta\leq\lambda$ of uncountable cofinality, then any pre-appropriate iteration $\Por\rest\la L,\bar{\Iwf}\ra$ forces $\add(\Nwf)\leq\theta_0$, $\cov(\Nwf)\leq\theta_1$, $\sfrak\leq\theta$, $\gfrak\leq\theta$, $\add(\Mwf)=\cof(\Mwf)=\mu$ and $\non(\Nwf)=\rfrak=\cfrak=\lambda$. If it is appropriate, equalities are forced for the first four cardinals and $\pfrak=\theta$.
\end{lemma}
\begin{proof}
   By the preservation Theorem \ref{ThmItPres}, $\Por\rest L$ is $\theta_0$-$\in^*$-good, $\theta_1$-$\pitchfork$-good and $\theta$-$\propto$-good. By hypotheses (I), (II) and (III), each respective family is preserved in the forcing extension, so they witness $\add(\Nwf)\leq\theta_0$, $\cov(\Nwf)\leq\bfrak_\pitchfork\leq\theta_1$ and $\sfrak\leq\theta$.

   For $\alpha\in L_H\cap\lambda\cdot\mu$ let $d_\alpha$ be the dominating real added at $\alpha$ and let $c_\alpha$ be the Cohen real added at $\alpha$ in the iteration (recall that Hechler forcing adds Cohen reals). As $L_\alpha\in\Iwf_\alpha$, then $d_\alpha$ is Hechler over $V^{\Por\rest L_\alpha}$ and $c_\alpha$ is Cohen over the same model. Therefore, $\{d_\alpha:\alpha\in L_H\cap\lambda\cdot\mu\}$ forms a scale of cofinality $\mu$ in $V^{\Por\rest L}$, so $\bfrak=\dfrak=\mu$ in that model (also use Lemma \ref{general_properties}). On the other hand, $\Por\frestr L$ forces $\non(\Mwf)\leq\mu\leq\cov(\Mwf)$ because of the $\mu$-cofinal Cohen reals added, so $\add(\Mwf)=\cof(\Mwf)=\mu$ is clearly forced.

   For $\alpha<\theta$ we put $W_\alpha=V^{\Por\upharpoonright Z_\alpha}$ and $W_\theta=V^{\Por\upharpoonright L}$ where $\la Z_\alpha\ra_{\alpha<\theta}$ is an increasing sequence of subsets of $L$ whose union is $L$ and $(Z_{\alpha+1}\menos Z_\alpha)\cap L_H\cap \lambda\cdot\mu\neq\emptyset$. As a consequence of Theorem \ref{newrealint}, $\la W_\alpha\ra_{\alpha\leq\theta}$ satisfies the hypothesis of Lemma \ref{lemmagfrak}, so $\gfrak\leq\theta$ holds in $V^{\Por\upharpoonright L}$.

   In $V^{\Por\rest L}$, it is clear that $\cfrak\leq\lambda$ because $\Por\rest L$ has size $L$. On the other hand, $\lambda\leq\dfrak_\pitchfork\leq\non(\Nwf)$ and $\lambda\leq\rfrak$ by Theorem \ref{dfrakbig} because $\Dor$ is $\pitchfork$-good and $\propto$-good.

   Now, if the iteration is appropriate, it further forces $\theta_0\leq\add(\Nwf)$, $\theta_1\leq\cov(\Nwf)$ and $\theta\leq\pfrak\leq\sfrak$ (recall that $\pfrak\leq\gfrak$). We show the second one (the others are proven similarly). In $V^{\Por\rest L}$, let $\Bwf$ be a family of Borel null sets of size $<\theta_1$ so there is a transitive model $N$ of a large enough fragment of ZFC such that $\Bwf\subseteq N$ (the Borel codes) and $|N|<\theta_1$. By (8) of Definition \ref{DefAppr}, there is an $x\in L_R$ such that $\Qor_x=\Bor^N$, so $\Qor_x$ has already added a random real over $N$ and $\Bwf$ does not cover that real.
\end{proof}

To prove the Main Theorem, we need to construct an appropriate iteration that forces $\afrak=\lambda$. The following lemma is essential to construct this iteration.

\begin{mainlemma}
Let $\theta \leq\delta <\lambda$. Let $\Por\rest \la L^\delta,\bar{\mathcal{I}}^\delta\ra$ be a pre-appropriate iteration and let $\dot{\mathcal{A}}$ be a $\Por\rest L^\delta$-name for an almost disjoint family such that $\theta^+\leq|\dot{\mathcal{A}}|<\lambda$. Then there is $\delta^\prime$, $\delta <\delta^\prime <\lambda$ and an appropriate iteration $\Por'=\Por'\rest \la L^{\delta^\prime},\bar{\mathcal{I}}^{\delta^\prime}\ra$ such that
\begin{enumerate}[(a)]
\item $\Por'\restrict L^\delta = \Por\restrict L^\delta$, and
\item $\Vdash_{\Por'\restrict L^{\delta^\prime}} `` \dot{\mathcal{A}}\;\hbox{is not maximal}"$.
\end{enumerate}
\end{mainlemma}
\begin{proof}
This proof is inspired by \cite[Thm 3.3]{brendle02}.

Let $\dot{\mathcal{A}}=\{\dot{a}_\epsilon:\epsilon <\nu\}$ for some $\theta^+\leq\nu<\lambda$ be a $\Por\restrict \la L^\delta,\mathcal{I}^\delta\ra$-name for an almost disjoint  family. For every $\dot{a}_\epsilon$ there is a $B_\epsilon \subseteq L^\delta$ of size $<\theta$ such that $\dot{a}_\epsilon$ is a $\Por\rest B_\epsilon$-name for a real. We may assume that $B_\epsilon$ is c.i.s.t..  Indeed, start with an arbitrary $B^0_\epsilon$ of size less than $\theta$ such that $\dot{a}_\epsilon$ is a $\Por\rest B^0_\epsilon$-name (by Lemma~\ref{general_properties} such $B_\epsilon^0$ exists) and, for $n\in\omega$, define $B^{n+1}_\epsilon$ as the closure of $B^n_\epsilon\cup\{C_x^\delta:x\in B^n_\epsilon\cap L^\delta_C\}$ under initial segments. Take $B_\epsilon=\bigcup_{n\in\omega} B^n_\epsilon$.

By the $\Delta$-system Lemma (because $\theta^{<\theta}=\theta$) we may assume that $\la B_\alpha:\alpha<\theta^+\ra$ is a $\Delta$-system with root $R$ which is also a c.i.s.t. (so $C^\delta_x\subseteq R$ for $x\in R\cap L^\delta_C$).  By Lemma \ref{notmanytypes} we thin out the $\Delta$-system so that, for all $\alpha\neq\beta$, there is a $\Por\restrict \la L^\delta,\mathcal{I}^\delta\ra$-isomorphism
$\phi_{\alpha,\beta}: B_\alpha\to B_\beta$, which lifts to an isomorphism $\Phi_{\alpha,\beta}: \Por\rest \la B_\alpha, \mathcal{I}\rest B_\alpha\ra\to \Por\rest\la B_\beta,\mathcal{I}\rest B_\beta\ra$ (see Definition~\ref{DefItIsom}).
Moreover we may assume that
\begin{itemize}
\item $\phi_{\alpha,\beta}\rest R$ is the identity map,
\item  $\phi_{\alpha,\beta}[B_\alpha\cap L^\delta_F]= B_\beta\cap L^\delta_F$, $\phi_{\alpha,\beta}[B_\alpha\cap L^\delta_R]= B_\beta\cap L^\delta_R$, $\phi_{\alpha,\beta}[B_\alpha\cap L^\delta_A]= B_\beta\cap L^\delta_A$,
\item if $x\in B_\alpha\cap L^\delta_F$ then $\Phi_{\alpha,\beta}$ sends $\dot{F}^\delta_x$ to $\dot{F}^\delta_{\phi_{\alpha,\beta}(x)}$ (recall that $\phi_{\alpha,\beta}[C^\delta_x]=C^\delta_{\phi_{\alpha,\beta}(x)}$),
\item $\Phi_{\alpha,\beta}$ sends $\dot{a}_\alpha$ to $\dot{a}_\beta$,
\item $\phi^{-1}_{\alpha,\beta}=\phi_{\beta,\alpha}$ and $\phi_{\beta,\gamma}\circ\phi_{\alpha,\beta}=\phi_{\alpha,\gamma}$, likewise for the induced isomorphisms.
\end{itemize}
By shrinking again, we also assume that there is a $\rho_0<\theta$ such that, for any $\alpha<\theta^+$, $x\in B_\alpha$ and $k<|x|$, if $x(k)$ is negative then $x(k)\in S_{\rho}$ for some $\rho<\rho_0$.

Let $T\subseteq L^{\theta,\theta}$ be a tree of size $<\theta$ that represents $\la B_\alpha\ra_{\alpha <\theta^+}$, that is, for each $\alpha<\theta^+$ there is a bijection $x_\alpha:T\to B_\alpha$ satisfying (i)-(viii) of Definition \ref{DefItIsom} and such that $\phi_{\alpha,\beta}\circ x_\alpha=x_\beta$ for any $\beta\neq\alpha$.

Let $S\subseteq T$ be a tree which represents the root of the $\Delta$-system, that is, $x_\alpha[S]=R$ for each $\alpha<\theta^+$. Note that for all $\alpha,\beta$ in $\theta^+$ and all $t\in S$ we have $x_\alpha(t)=x_\beta(t)$. Furthermore we may assume that whenever $s\in S\cup\{\emptyset\}$ and $t:=s^\smallfrown \la \xi\ra\in T\backslash S$ then, for all $\alpha<\theta^+$, we have that
\begin{itemize}
\item $x_\alpha(t)(|s|)>\theta$, in case $\xi$ is positive, and
\item $x_\alpha(t)(|s|) <\theta^*$, in case $\xi$ is negative.
\end{itemize}

Now, let $\{ t_\eta: \eta < \kappa\}$ with $\kappa<\theta$ enumerate $\{s^\smallfrown \la \xi\ra: s\in S\cup\{\emptyset\}, s^\smallfrown \la \xi\ra\in T\backslash S\}$. Consider the coloring
$F: [\theta^+]^2\to\kappa$ defined as follows:
for $\alpha<\beta$ let
\begin{multline*}
  F(\alpha,\beta)=\min\{\eta<\kappa: \textrm{either }x_\alpha(t_\eta)(|t_\eta|-1)>x_\beta(t_\eta)(|t_\eta|-1)\textrm{\ and }t_\eta(|t_\eta|-1)\textrm{\ is positive,}\\
  \textrm{or }x_\alpha(t_\eta)(|t_\eta|-1)<x_\beta (t_\eta)(|t_\eta|-1)\textrm{\ and }t_\eta(|t_\eta|-1)\textrm{\ is negative}\}
\end{multline*}
when such a $\eta$ exists, otherwise put $F(\alpha,\beta)=0$.

We will use the following reformulation of Erd\"os-Rado theorem.
\begin{clm} If $F:[\theta^+]^2\to \kappa$, where $\kappa<\theta$ and $\theta^{<\theta}=\theta$, then there is a homogeneous set of  size $\theta$.
\end{clm}
\begin{proof}
   Similar to~\cite[Lemma III.8.11]{kunen11}.
\end{proof}

Thus we can find an $F$-homogeneous set of size $\theta$. It should have color $0$ since otherwise we will have an infinite decreasing chain of ordinals. Without loss of generality, this homogeneous set is $\theta$.

For every $s\in S\cup\{\emptyset\}$, $\xi$ and $\eta$ such that $s^\smallfrown \la \xi\ra$, $s^\smallfrown \la \eta\ra$ are in $T\backslash S$, denote by $\upsilon_{s^\smallfrown\la\xi\ra}$ the limit of $\{ x_\alpha (s^\smallfrown \la\xi\ra)(|s|)\}$ (which is a supremum if $\xi$ is positive, or an infimum otherwise). We may assume the following:

\begin{itemize}
\item if $\xi<\eta$ are positive, then
\begin{itemize}
\item either $\upsilon_{s^\smallfrown\la\xi\ra}< x_0(s^\smallfrown \la\eta\ra)(|s|)$ (when $\upsilon_{s^\smallfrown\la\xi\ra}<\upsilon_{s^\smallfrown\la\eta\ra}$),
\item or $x_\alpha(s^\smallfrown\la \eta\ra)(|s|) < x_\beta(s^\smallfrown \la\xi\ra)(|s|)$ for all $\alpha<\beta<\theta$ (when $\upsilon_{s^\smallfrown\la\xi\ra}=\upsilon_{s^\smallfrown\la\eta\ra}$).\footnote{In that case, there is a club subset of $\theta$ with that property (for fixed $s^\smallfrown \la \xi\ra,s^\smallfrown \la \eta\ra$).}
\end{itemize}
\item if $\xi<\eta$ are negative, then
\begin{itemize}
\item either $x_0(s^\smallfrown \la\xi\ra)(|s|)<\upsilon_{s^\smallfrown\la\eta\ra}$,
\item or $x_\alpha(s^\smallfrown \la \xi\ra)(|s|)> x_\beta (s^\smallfrown \la \eta\ra)(|s|)$ for all $\alpha<\beta<\theta$.
\end{itemize}
\end{itemize}

Recall that any object in $L^\delta$ contains only elements of $(\delta^*,\delta)$ from the second coordinate on. Now, choose $\gamma^*\in S_{\rho_0}$ such that $\delta<\gamma<\lambda$ (exists because $S_{\rho_0}$ is co-initial in $\lambda^*$) and let $\delta'<\lambda$ be any ordinal larger than $\gamma$ (we can also allow for $\delta^\prime$ to be a successor ordinal).

We define $x_\nu:T\to L^{\delta'}$ as follows.

\begin{itemize}
\item if $t\in S$ then $x_\nu(t)=x_0(t)\in R$,
\item if $t=s^\smallfrown \la\xi\ra\in T\backslash S$ with $s\in S\cup\{\emptyset\}$, then
\begin{itemize}
\item if $\xi$ is positive, then
\[
x_\nu(t)=
\begin{cases}
 x_\nu(s)^\smallfrown \la\upsilon_{s^\smallfrown\la\xi\ra},\gamma^*\ra^\smallfrown x_0(t)(|s|), & \text{if }|s|\neq0 \\
 x_\nu(s)^\smallfrown \la\upsilon_{s^\smallfrown\la\xi\ra},\gamma^*\ra^\smallfrown \la \xi\ra,                         & \text{if } |s|=0
\end{cases}
\]
\item if $\xi$ is negative, then $x_\nu(t)=x_\nu(s)^\smallfrown\la \upsilon_{s^\smallfrown\la\xi\ra},\gamma\ra^\smallfrown x_0(t)(|s|)$.
\end{itemize}
\item if $t\in T$ then
$x_\nu(t)=x_\nu(t\rest m)^\smallfrown x_0(t)\rest [m, |t|)$ were $m$ is the minimal (if exists) such that $t\rest m\in T\backslash S$,.
\end{itemize}

Put $B_\nu=\{x_\nu(t):t\in T\}$, which is a subset of $L^{\delta'}$ that is isomorphic (as a linear order) with $T$ via $x_\nu$. Thus, $\phi_{\alpha,\nu}:=x_\nu\circ x_\alpha^{-1}:B_\alpha\to B_\nu$ is an order isomorphism for all $\alpha<\theta$. Moreover, $B_\nu\cap L^\delta=R$ and $\phi_{\alpha,\nu}\rest R$ is the identity map. Let $\phi_{\nu,\alpha}=\phi^{-1}_{\alpha,\nu}$. Note that $\phi_{0,\nu}$ is also a template-isomorphism (see Defintion \ref{DefTempIsom}) between $\langle B_0,\bar{\Iwf}^\delta\rest B_0\rangle$ and $\la B_\nu,\bar{\Jwf}\ra$ where $\Jwf_z=\{\phi_{0,\nu}[X]:X\in\Iwf^\delta_{\phi_{\nu,0}(z)}\rest B_0\}$ for each $z\in B_\nu$.

\begin{claim}\label{claimInnoc1}
  $\la B_\nu,\Iwf^{\delta'}\rest B_\nu\ra$ is a $\theta$-innocuous extension of $\la B_\nu,\bar{\Jwf}\ra$.
\end{claim}
\begin{proof}
   Similar to the argument in \cite[Thm 3.3]{brendle02}.
\end{proof}

Let
$$L^{\delta^\prime}_F:= L^\delta_F\cup \phi_{0,\nu}[ L^\delta_F\cap B_0]\cup\{\la\eta,\gamma,\gamma,\gamma\ra:\eta\in\lambda\cdot\mu,\ \eta\equiv0\mod3\},$$
$$L^{\delta^\prime}_R:= L^\delta_R\cup \phi_{0,\nu}[ L^\delta_R\cap B_0]\cup\{\la\eta,\gamma,\gamma,\gamma\ra:\eta\in\lambda\cdot\mu,\ \eta\equiv1\mod3\},$$
$$L^{\delta^\prime}_A:= L^\delta_A\cup \phi_{0,\nu}[ L^\delta_A\cap B_0]\cup\{\la\eta,\gamma,\gamma,\gamma\ra:\eta\in\lambda\cdot\mu,\ \eta\equiv2\mod3\},$$
let $L^{\delta'}_C=L^{\delta^\prime}_A\cup L^{\delta'}_R\cup L^{\delta'}_F$ and $L^{\delta^\prime}_H= L^{\delta^\prime}\backslash L^{\delta^\prime}_C$ which contains $L^\delta_H$. Fix a bijection $g:\lambda\to\lambda\times\theta$ and an enumeration $\{C_{\zeta,\beta}:\beta<\lambda\}$ of $[L^{\delta'}_{\lambda\cdot\zeta}]^{<\theta}$ (which is a subset of $\hat{\Iwf}^{\delta'}_{\lambda\cdot\zeta}$) for each $\zeta<\mu$. When $z$ is an ordered pair, $(z)_0$ denotes its first coordinate and $(z)_1$ its second.

For $x\in L^{\delta'}_C$,
\begin{itemize}
\item if $x\in L^{\delta}_C$, then let $C^{\delta^\prime}_x:= C^\delta_x$;
\item if $x=x_\nu(t)$ for some $t\in T$ let $C^{\delta^\prime}_x:=\phi_{0,\nu}[C^\delta_{x_0(t)}]$ (note that this does not disagree with the previous bullet);
\item if $x=\la \eta,\gamma,\gamma,\gamma\ra$ and $\eta=\lambda\cdot\zeta+3\cdot\varrho+i$ where $\zeta<\mu$, $\varrho<\lambda$ and $i<3$, let $C^{\delta'}_x=C_{\zeta,(g(\varrho))_0}$.
\end{itemize}
Note that, for $\alpha\leq\nu$, if $x\in B_\alpha\cap L^{\delta'}_C$ then $C^{\delta'}_x\subseteq B_\alpha$.

We construct a $(\la\Dor\ra,\theta)$-standard iteration $\Por'\rest\la L^{\delta^\prime},\bar{\mathcal{I}}^{\delta^\prime}\ra$ such that
\begin{enumerate}[(i*)]
\setcounter{enumi}{0}
   \item $L^{\delta'}_H$ are the coordinates where (full) Hechler forcing is used, while $L^{\delta'}_C$ are the coordinates where ccc posets of size $<\theta$ are used according to what we naturally mean for coordinates in $L^{\delta'}_A$ (localization poset), $L^{\delta'}_R$ (random) and in $L^{\delta'}_F$ (Mathias-Prickry);
   \item for $X\subseteq L^\delta$, $\Por'\rest X=\Por\rest X$;
   \item for $z\in L^{\delta}_C$, $\Qnm'_z=\Qnm_z$;
   \item there is a forcing isomorphism $\Phi_{\nu,0}:\Por'\rest B_\nu\to\Por\rest B_0$ that lifts $\phi_{\nu,0}$ (in the sense of (viii*) and (ix*) below) and $\Phi_{\nu,0}\rest(\Por\rest R)$ is the identity;
   \item for each $\zeta<\mu$ and $\beta<\lambda$, $\{\dot{F}'_{\zeta,\beta,\alpha}:\alpha<\theta\}$ enumerates\footnote{This family of names has size $\leq\theta$ because $|\Por'\rest C_{\zeta,\beta}|\leq\theta$ as noted in the proof of Lemma \ref{notmanytypes}.} all (nice) $\Por'\rest C_{\zeta,\beta}$-names for filter bases of size $<\theta$ and, if $\eta=\lambda\cdot\zeta+3\cdot\varrho$ for some $\varrho<\lambda$ and if $z=\la\eta,\gamma,\gamma,\gamma\ra$ then $\Qnm_z=\Mor_{\dot{F}'_{\zeta,g(\varrho)}}$;
   \item for each $\zeta<\mu$ and $\beta<\lambda$, $\{\dot{\Bor}_{\zeta,\beta,\alpha}:\alpha<\theta\}$ enumerates all (nice) $\Por'\rest C_{\zeta,\beta}$-names for subalgebras of $\Bor$ of size $<\theta_1$ and, if $\eta=\lambda\cdot\zeta+3\cdot\varrho+1$ for some $\varrho<\lambda$ and if $z=\la\eta,\gamma,\gamma,\gamma\ra$ then $\Qnm_z=\Bnm_{\zeta,g(\varrho)}$;
   \item for each $\zeta<\mu$ and $\beta<\lambda$, $\{\dot{\Loc}_{\zeta,\beta,\alpha}:\alpha<\theta\}$ enumerates all (nice) $\Por'\rest C_{\zeta,\beta}$-names for subposets of $\Loc$ of size $<\theta_0$ and, if $\eta=\lambda\cdot\zeta+3\cdot\varrho+2$ for some $\varrho<\lambda$ and if $z=\la\eta,\gamma,\gamma,\gamma\ra$ then $\Qnm_z=\Locnm_{\zeta,g(\varrho)}$.
\end{enumerate}
Conditions (v*), (vi*) and (vii*) guarantee that $\Por'\rest\la L^{\delta^\prime},\mathcal{I}^{\delta^\prime}\ra$ is an appropriate iteration. For instance, if $\Qnm$ is a $\Por'\rest L^{\delta'}$-name for a subalgebra of random forcing of size $<\theta_1$, by Lemma \ref{general_properties} there exists $C'\subseteq L^{\delta'}$ of size $<\theta$ such that $\Qnm$ is (forced to be equal to) a $\Por'\rest C'$-name, so there is $\zeta<\mu$ such that $C'\subseteq L^{\delta'}_{\lambda\cdot\zeta}$ and there exists a $\beta<\lambda$ such that $C'=C_{\zeta,\beta}$. By (vi*), $\Qnm=\Bnm_{\zeta,\beta,\alpha}$ for some $\alpha<\theta$ so $\Qnm=\Qnm_z$ where $z=\la\eta,\gamma,\gamma,\gamma\ra$, $\eta=\lambda\cdot\mu+3\cdot\varrho+1$ and $\varrho=g^{-1}(\beta,\alpha)$.

By Claim \ref{claimInnoc1} and Lemmas \ref{InnConstr} and \ref{InnEqv}, there is a $(\la\Dor\ra,\theta)$-template iteration $\Por'\rest\la B_\nu,\bar{\Iwf}^{\delta'}\rest B_\nu\ra$ and a forcing isomorphism $\Phi_{\nu,0}:\Por'\rest B_\nu\to\Por\rest B_0$ satisfying
\begin{enumerate}[(i*)]
\setcounter{enumi}{7}
   \item $\Phi_{\nu,0}\rest(\Por'\rest \phi_{0,\nu}[X]):\Por'\rest \phi_{0,\nu}[X]\to\Por\rest X$ is an isomorphism for any $X\subseteq B_0$ and
   \item $\Qnm_{\phi_{0,\nu}(x)}$ is the $\Por'\rest C^{\delta'}_{\phi_{0,\nu}(x)}$-name associated to $\Qnm_x$ via $\Phi_{\nu,0}$ for any $x\in B_0$.
\end{enumerate}
It is clear that $\Por'\rest R=\Por\rest R$ and $\Phi_{\nu,0}\rest(\Por\rest R)$ is the identity map. Therefore, as $L^\delta\cap B_\nu=R$, we can easily extend $\Por'\rest\la B_\nu,\bar{\Iwf}^{\delta'}\rest B_\nu\ra$ to an iteration $\Por'\rest\la L^\delta\cup B_\nu,\bar{\Iwf}^{\delta'}\rest(L^\delta\cup B_\nu)\ra$ satisfying (i*). Furthermore, as $(L^\delta\cup B_\nu)\cap\{\la\eta,\gamma,\gamma,\gamma\ra:\eta\in\lambda\cdot\mu\}=\emptyset$, we can extend the iteration to $\Por'\rest\la L^{\delta^\prime},\bar{\mathcal{I}}^{\delta^\prime}\ra$ satisfying, additionally, (v*)-(vii*). Observe that, for any $0<\alpha<\nu$, $\Phi_{\nu,\alpha}:=\Phi_{0,\alpha}\circ\Phi_{\nu,0}:\Por'\rest B_\nu\to\Por\rest B_\alpha$ is a forcing isomorphism that lifts $\phi_{\nu,\alpha}$ and satisfies similar properties as (iv*), (viii*) and (ix*).

\bigskip
Let $\dot{a}_\nu$ be the $\Por'\rest B_\nu$-name corresponding to $\dot{a}_0$ via $\Phi_{\nu,0}$.  To finish the proof, we show that $\Vdash_{\Por'\rest L^{\delta^\prime}} \forall\epsilon <\nu (\dot{a}_\epsilon\cap\dot{a}_\nu\;\hbox{is finite})$. Fix $\epsilon <\nu$. As $|B_\epsilon|<\theta$ and $\la B_\alpha:\alpha<\theta\ra$ forms a $\Delta$-system, there is an $\alpha_\epsilon <\theta$ such that $\forall \alpha\in[\alpha_\epsilon,\theta) (B_\alpha\cap B_\epsilon\subseteq R)$. Moreover, we may assume that
\begin{itemize}
  \item[(**)] For any $s\in S\cup\{\emptyset\}$ and $t=s^\smallfrown \la\xi\ra\in T\backslash S$, if $\xi$ is positive then
   $$\sup\{ y(|s|): y\in B_\epsilon, y\rest |s| = x_\nu (s)\;\hbox{and}\; y(|s|)< x_\nu(|s|)\} < x_{\alpha_\epsilon}(|s|)$$
   and if $\xi$ is negative then
   $$\inf\{ y(|s|): y\in B_\epsilon, y\rest |s| = x_\nu (s)\;\hbox{and}\; y(|s|)> x_\nu(|s|)\} > x_{\alpha_\epsilon}(|s|).$$
\end{itemize}

Take any $\alpha\in[\alpha_\epsilon,\theta)\menos\{\epsilon\}$ and consider the mapping $\phi: B_\nu\cup B_\epsilon\to B_\alpha\cup B_\epsilon$ where
\[x\mapsto \phi(x)=
\begin{cases}
\phi_{\nu,\alpha}(x), & \text{if } x\in B_\nu, \\
x,                     & \text{if }x\in B_\epsilon.
\end{cases}
\]

From (**), $\phi:\la B_\nu\cup B_\epsilon,\bar{\Jwf}'\ra\to\langle B_\alpha\cup B_\epsilon,\bar{\Iwf}^{\delta'}\rest(B_\alpha\cup B_\epsilon)\rangle$ is a template isomorphism where $\Jwf'_z=\{\phi^{-1}[X]:X\in\Iwf^{\delta}_{\phi(z)}\rest(B_\alpha\cup B_\epsilon)\}$ for any $z\in B_\nu\cup B_\epsilon$. Furthermore,

\begin{claim}\label{claimInnoc2}
   $\la B_\nu\cup B_\epsilon,\bar{\Iwf}^{\delta'}\rest(B_\nu\cup B_\epsilon)\ra$ is a $\theta$-innocuous extension of $\la B_\nu\cup B_\epsilon,\bar{\Jwf}'\ra$.
\end{claim}
\begin{proof}
   Similar to the argument in \cite[Thm 3.3]{brendle02}.
\end{proof}

Therefore, by Lemma \ref{InnEqv} and items (i*)-(iv*), (viii*) and (ix*), there is a forcing isomorphism $\Phi:\Por'\rest(B_\nu\cup B_\epsilon)\to\Por\rest(B_\alpha\cup B_\epsilon)$ lifting $\phi$, moreover, $\Phi\rest(\Por'\rest B_\nu)=\Phi_{\nu,\alpha}$ and $\Phi\rest(\Por'\rest B_\epsilon)$ is the identity map (these by uniqueness in Lemma \ref{InnEqv}) so $\dot{a}_\nu$ is identified with $\dot{a}_\alpha$ via $\Phi$ and $\dot{a}_\epsilon$ is identified with itself. As $\Vdash_{\Por\rest(B_\alpha\cup B_\epsilon)}|\dot{a}_\alpha\cap\dot{a}_{\epsilon}|<\aleph_0$ we conclude that $\Vdash_{\Por'\rest(B_\nu\cup B_\epsilon)}|\dot{a}_\nu\cap\dot{a}_{\epsilon}|<\aleph_0$.
\end{proof}

As a consequence of the previous proof, we obtain

\begin{corollary}\label{preapprExt}
   Let $\delta <\lambda$ and $\Por\rest \la L^\delta,\bar{\mathcal{I}}^\delta\ra$ a pre-appropriate iteration. Then there is $\delta^\prime$, $\delta <\delta^\prime <\lambda$ and an appropriate iteration $\Por'=\Por'\rest \la L^{\delta^\prime},\bar{\mathcal{I}}^{\delta^\prime}\ra$ such that $\Por'\restrict L^\delta = \Por\restrict L^\delta$.
\end{corollary}
\begin{proof}
   Choose any $\gamma$, $\delta<\gamma<\lambda$ and let $\delta'$ be any ordinal strictly between $\gamma$ and $\lambda$. $\Por'=\Por'\rest \la L^{\delta^\prime},\bar{\mathcal{I}}^{\delta^\prime}\ra$ is defined exactly as in the previous proof (just ignore anything related to $B_\nu$, $B_0$, $\phi_{0,\nu}$ and $\Phi_{\nu,0}$).
\end{proof}


\begin{proof}[Proof of the Main Theorem]
Fix a bookkeeping function $h: \lambda\to\lambda\times\lambda$ such that $h$ is a bijection and for all $\alpha\in\lambda$ if $h(\alpha)=(\xi,\eta)$ then $\alpha\geq\xi$. By recursion, we define a sequence $\la\Por^\alpha\rest\la L^{\delta_\alpha},\bar{\Iwf}^{\delta_\alpha}\ra\ra_{\alpha\leq\lambda}$ of appropriate iterations as follows.

\bigskip
\noindent
{\emph{Basic step $\alpha=0$.}} Let $\Por\rest\la L^{\theta^+},\bar{\Iwf}^{\theta^+}\ra$ be a pre-appropriate iteration with $L^{\theta^+}_H=L^{\theta^+}$ (that is, $\Dor$ is used everywhere). By Corollary \ref{preapprExt}, find $\delta_0\in(\theta^+,\lambda)$ and an appropriate iteration $\Por^0\rest\la L^{\delta_0},\bar{\Iwf}^{\delta_0}\ra$.

\bigskip
\noindent
{\emph{Successor step.}} Let  $\la \dot{\mathcal{A}}_{\alpha,\eta}:\eta<\lambda\ra$ enumerate all (nice) $\Por^\alpha\rest L^{\delta_\alpha}$-names of almost disjoint families of size in $[\theta^+,\lambda)$ (such enumeration has size $\lambda$ because $\lambda^{<\lambda}=\lambda$ and $|\Por^{\delta_\alpha}\rest L^{\delta_\alpha}|=\lambda$). By the Main Lemma, we can find $\delta_{\alpha+1}\in(\delta_\alpha,\lambda)$ and an appropriate iteration $\Por^{\alpha+1}\rest \la L^{\delta_{\alpha+1}},\bar{\mathcal{I}}^{\delta_{\alpha+1}}\ra$ such that $\Vdash_{\Por^\alpha\rest L^{\delta_{\alpha+1}}} ``\dot{\mathcal{A}}_{h(\alpha)}\;\text{is not maximal}"$ ($\dot{\Awf}_{h(\alpha)}$ has already been defined because $\xi\leq\alpha$ when $h(\alpha)=(\xi,\eta)$).

\bigskip
\noindent
{\emph{Limit step}}. Let $\delta=\sup_{\xi<\alpha} \{\delta_\xi\}$ so $L^\delta=\bigcup_{\xi<\alpha} L^{\delta_\xi}$. If $\alpha<\lambda$ then $\delta<\lambda$, but $\alpha=\lambda$ implies $\delta=\lambda$. $\mathcal{I}^{\delta_\xi}_x=\mathcal{I}^\delta_x\rest L^{\delta_\xi}$ for any $x\in L^{\delta_\xi}$ and $\xi<\alpha$ by Lemma \ref{restTemp}. Let $L_H^\delta=\bigcup_{\xi<\alpha} L_H^{\delta_\xi}$, $L^\delta_F=\bigcup_{\xi<\alpha} L_F^{\delta_\xi}$ and likewise for $L_R^\delta$ and $L_A^\delta$. In addition, for every $x\in L^\delta_C$ we can find $\xi<\alpha$ such that $x\in L^{\delta_\xi}_C$. Then define $C^\delta_x= C^{\delta_\xi}_x$ and $\Qnm^\delta_x=\Qnm^{\delta_\xi}_x$, which does not depend on the choice of $\xi$. This allows us to define a pre-appropriate iteration $\hat{\Por}\rest\la L^\delta,\bar{\mathcal{I}}^\delta\ra$ such that $\hat{\Por}\rest L^{\delta_\xi} = \Por^\xi\rest L^{\delta_\xi}$ for any $\xi<\alpha$. It is clear that the iteration $\hat{\Por}\rest\la L^\delta,\bar{\mathcal{I}}^\delta\ra$ is appropriate when $\mathrm{cf}(\alpha)\geq\theta$, in which case $\delta_\alpha=\delta$ and $\Por^\alpha\rest\la L^{\delta_\alpha},\bar{\mathcal{I}}^{\delta_\alpha}\ra=\hat{\Por}\rest\la L^\delta,\bar{\mathcal{I}}^\delta\ra$, moreover, this is the direct limit of $\Por^{\xi}\rest L^{\delta_\xi}$ for $\xi<\alpha$ since any condition $p\in \Por^\alpha\rest L^{\delta_\alpha}$ is restricted to a subset of size $<\theta$ by Lemma \ref{general_properties}; if $\mathrm{cf}(\alpha)<\theta$ we just find $\delta_\alpha\in(\delta,\lambda)$ and an appropriate iteration $\Por^\alpha\rest\la L^{\delta_\alpha},\bar{\Iwf}^{\delta_\alpha}\ra$ such that $\Por^\alpha\rest L^{\delta}=\hat{\Por}\rest L^\delta$ by Corollary \ref{preapprExt}.

\bigskip
As $\Por^\lambda\rest \la L^\lambda,\mathcal{I}^\lambda\ra$ is an appropriate iteration, by Lemma \ref{ApprLemma} we only need to show that $\Vdash_{\Por^\lambda\rest L^\lambda}\mathfrak{a}\notin[\theta^+,\lambda)$ (because $\Por^\lambda\rest L^\lambda$ already forces $\bfrak=\mu\geq\theta^+$ and $\bfrak\leq\afrak$ is probable in ZFC). Let $\dot{\Awf}$ be a $\Por^\lambda\rest L^\lambda$-name for an almost disjoint family of size in $[\theta^+,\lambda)$ (by ccc-ness, this size can be decided). As $\Por^\lambda\rest L^\lambda$ is the direct limit of $\Por^{\delta_\alpha}\rest L^{\delta_\alpha}$ for $\alpha<\lambda$, we can find $\xi,\eta<\lambda$ such that $\dot{\Awf}=\dot{\Awf}_{\xi,\eta}$ so, if $h(\alpha)=(\xi,\eta)$ then $\Por^{\alpha+1}\rest L^{\delta_{\alpha+1}}$ already forces that $\dot{\Awf}$ is not maximal.
\end{proof}



\section{Questions}\label{SecDisc}

J. Brendle~\cite{brendle03} modified Shelah's original template iteration technique to incorporate a product-like forcing as a complete suborder of the entire template iteration. This modified template iteration produces the consistency of $\mathfrak{a}$ being of countable cofinality.
Recently, the first author jointly with A. T\"ornquist (see~\cite{VFAT}) showed that the minimal size of a maximal cofinitary group $\mathfrak{a}_g$, as well as some other close combinatorial relatives of the almost disjointness number, like $\mathfrak{a}_p$, $\mathfrak{a}_e$, can be of countable cofinality. Of interest remains the following question:

\begin{question}
Can the iteration techniques developed in this paper be further developed to expand the results by including the case in which $\mathfrak{a}$, $\mathfrak{a}_g$, $\mathfrak{a}_p$ or $\mathfrak{a}_e$ are singular, or even of countable cofinality?
\end{question}

The iteration of eventually different forcing along Shelah's original template produces the consistency of $\mathfrak{a}=\aleph_1<\hbox{non}(\mathcal{M})<\mathfrak{a}_g$ (see~\cite[Thm. 4.11]{brendle02}). It is unknown whether this consistency result could be improved as follows.

\begin{question}
Is it consistent that $\aleph_1<\mathfrak{a}<\hbox{non}(\mathcal{M})<\mathfrak{a}_g$?
\end{question}

In his work on template iterations, Shelah \cite{shelah04} (see also \cite{brendle07}) also constructed, using a measureable cardinal $\kappa$, a ccc poset that forces $\kappa<\ufrak<\afrak$. As this poset also forces $\ufrak=\bfrak=\sfrak$, the consistency of $\bfrak=\sfrak<\afrak$ is clear modulo a measurable cardinal. However, it is not known whether these consistency results can be obtained from ZFC alone.

\begin{question}\label{Qb=s<a}
   Is it consistent with ZFC alone that
   \begin{enumerate}[(1)]
      \item $\bfrak=\sfrak<\afrak$?
      \item $\bfrak=\sfrak=\aleph_1<\afrak=\aleph_2$? (see \cite{BrRa}).
      \item $\ufrak<\afrak$?
   \end{enumerate}
\end{question}

Question \ref{Qb=s<a}(2) is a very important and challenging problem. It is closely related to the famous Roitman's problem (still open) on whether ``$\dfrak=\aleph_1$ implies $\afrak=\aleph_1$" is provable in ZFC.




\end{document}